\documentclass[twoside,11pt]{article}
\usepackage{xstring}
\usepackage[english]{babel}
\usepackage{amsmath}
\usepackage{amsthm}
\usepackage{cases}
\usepackage{amsfonts}
\usepackage{amssymb}
\usepackage{graphicx}
\usepackage{colortbl,dcolumn}
\usepackage{float}
\usepackage{paralist}  
\allowdisplaybreaks[4]

\newtheorem{lemma}{\bf Lemma}[section]
\newtheorem{theorem}{\bf Theorem}[section]
\newtheorem{proposition}{\bf Proposition}[section]

\newtheorem{remark}{\bf Remark}[section]

\numberwithin{equation}{section}

\begin{document}
	
	\title{{\Large Transition layers to chemotaxis-consumption models with volume-filling effect}
		\footnotetext{\small
			*Corresponding author.}
		\footnotetext{\small E-mail addresses: lixw928@nenu.edu.cn (XL), lijy645@nenu.edu.cn (JL).} }
	
	\author{{Xiaowen Li and Jingyu Li$^\ast$}\\[2mm]
		\small\it School of Mathematics and Statistics, Northeast Normal University,\\
		\small\it   Changchun 130024, P.R.China}

	\date{}
	
	\maketitle
	
	\begin{quote}
		\small \textbf{Abstract}: We are interested in the dynamical behaviors of solutions to a parabolic-parabolic chemotaxis-consumption model with a volume-filling effect on a bounded interval, where the physical no-flux boundary condition for the bacteria and mixed Dirichlet-Neumann boundary condition for the oxygen are prescribed. By taking a continuity argument, we first show that the model admits a unique nonconstant steady state. Then we use Helly's compactness theorem to show that the asymptotic profile of steady state is a transition layer as the chemotactic coefficient goes to infinity. Finally, based on the energy method along with a cancellation structure of the model, we show that the steady state is nonlinearly stable under appropriate perturbations. Moreover, we do not need any assumption on the parameters in showing the stability of steady state.
		
		
		\indent \textbf{AMS (2020) Subject Classification}: 35Q92, 35B35, 92C17


			

	\end{quote}

	
\section{Introduction}

In order to prevent the cell overcrowding that is undesirable from a biological standpoint, Hillen and Painter \cite{Hillen-Painter01,Hillen-Painter02} proposed the following chemotaxis model with a volume-filling effect:
\begin{equation}\label{1.1}
	\begin{cases}u_t=\nabla\cdot\left[\left(q(u)-q^{\prime}(u) u\right) \nabla u-\chi q(u) u \nabla v\right]+f(u,v),
		\\ v_t=\Delta v+g(u,v),\end{cases}
\end{equation}
where $u$ is the density of cell population, and $v$ is the concentration of chemical signal. The positive parameter $\chi$ represents the chemotactic coefficient that measures the strength of chemoattractants. $q(u)$ denotes the probability of cells finding space at their neighbouring locations, which generally satisfies the following condition: there is a maximal cell number $K$, called crowing capacity, such that
\begin{equation*}
	q(u)\equiv0,\	\forall u\geq K\text{ and }
	0< q(u)\leq1, \ q'(u)\leq0, \ \forall \ 0\leq u< K.
\end{equation*}
$f(u,v)$ and $g(u,v)$ represent the cell and chemical kinetics, respectively. Typical candidates of $g(u,v)$ include $g(u,v)=u-v$ and $g(u,v)=-uv$. In the former case the signal is produced by the cells, and the system is a chemotaxis-production model. In the latter case the signal is consumed by the cells, and the system is a chemotaxis-consumption model. In contrast to the classical chemotaxis model introduced by Keller and Segel \cite{KS}, the system \eqref{1.1} does not treat cells as point masses but takes into account the finite size of cells.

There are fruitful analytical works for the system \eqref{1.1} when the chemical signal is produced by the cells, i.e.  $g(u,v)=u-v$. If there is no growth term, i.e. $f(u,v)=0$, the system is relatively easy to handle and there are many profound results. In the case $q(u)=1-u$, Hillen and Painter \cite{Hillen-Painter01} first proved the global existence of classical solutions on a compact Riemannian manifold without boundary. Wrzosek \cite{Wrzosek04} generalized the works of  \cite{Hillen-Painter01} to weak solutions on a bounded domain with general boundary conditions, and showed that, there is a global attractor by using the dynamical system approach. Wrzosek \cite{Wrzosek06} further studied the large time behaviors of solutions by constructing Lyapunov functional. Subsequently, Jiang and Zhang \cite{Jiang-Zhang} showed that every solution of the system converges to an equilibrium as the time goes to infinity by using a non-smooth version of the {\L}ojasiewicz-Simon inequality. Wang \emph{et al.} \cite{Wang-Winkler-Wrzosek11,Wang-Winkler-Wrzosek12} further considered the general volume-filling effect, and discovered a critical parameter condition to separate the global existence and formation of singularity, where the singularity means the solution attains the value $1$ in either finite or infinite time. The volume-filling mechanism excludes the blow-up behaviors of cell density but usually generates complex biological patterns. Potapov and Hillen \cite{Potapov-Hillen05} performed the local bifurcation analysis to investigate the pattern formation in 1D. Wang and Xu \cite{Wang-Xu13} generalized the work of  \cite{Potapov-Hillen05} through a global bifurcation analysis. In fact, the global bifurcation theory was initially used in \cite{Wang2000} to prove the existence of spiky solutions of the chemotaxis models. Moreover, \cite{Wang-Xu13} also showed that, in contrast to the classical Keller-Segel model where the one is a spike, the asymptotic profile of steady state of  the system \eqref{1.1} with $g(u,v)=u-v$, $q(u)=1-u$ and $f(u,v)=0$ is indeed a transition layer in the sense that the limit of cell density is a step function as $\chi\rightarrow\infty$. If the logistic growth of cells is taken into account, i.e. $f(u,v)=u(1-u)$, the system is very difficult to analyze and there are only a few results achieved. Ma \emph{et al.} \cite{Ma-Ou-Wang12} employed the index theory together with the maximum principle to show that under the effect of growth the system has nonconstant steady states as $\chi$ grows. Ma \emph{et al.} \cite{Ma-Gao-Tong-Han,Ma-Wang15,Ma-Wang17} further generalized the works of  \cite{Ma-Ou-Wang12} to models with general volume-filling effect. Furthermore, \cite{Ma-Wang17,Hillen-Painter02,Wang-Hillen07} numerically predicted that under the effect of growth, the system generates spatio-temporal patterns and even chaotic patterns. Besides the spatial patterns, Ou and Yuan \cite{Ou-Yuan09} proved the existence of traveling front if the chemotactic coefficient is small. Although the system admits various patterns, it is usually challenging  to prove their stability. The only attempt was carried out by Lai \emph{et al.} \cite{Lai16}, where it was shown that the transition layer obtained by Wang and Xu \cite{Wang-Xu13} is nonlinearly stable under appropriate perturbations if $\chi\gg1$. For more works on the chemotaxis-production models with volume-filling effects, we refer the interested readers to \cite{Hillen-Painter09,Wrzosek10}.

In many experiments, biological patterns are formed by bacteria stimulated by the nutrient rather than the chemical signal released by themselves. For instance, E. coli in a capillary tube move as a traveling band by consuming oxygen \cite{Adler}, aerobic bacteria accumulate around the surface of a water drop to form a plume pattern by consuming oxygen \cite{Tuval}. Compared with the chemotaxis-production model, the volume-filling effect on the chemotaxis-consumption model is much less understood. The purpose of this paper is to investigate the dynamics of the chemotaxis-consumption model with volume-filling effect. The mathematical model is the following
\begin{equation}\label{original problem}
	\begin{cases}u_t=\left[\left(q(u)-q^{\prime}(u) u\right) u_x-\chi q(u) u v_x\right]_x, & x \in(0,L), \\ v_t=v_{x x}-u v, & x \in(0,L),\\ \left( u,v\right)\left(x_0,0\right)=\left(u_0,v_0\right)(x),& x \in(0,L),\end{cases}
\end{equation}
where $u$ is the density of bacteria, and $v$ is the concentration of oxygen. The boundary conditions of \eqref{original problem} are given by
\begin{equation}\label{original boundary condition}
	\begin{cases}
		\left(q(u)-q^{\prime}(u) u\right) u_x-\chi q(u) u v_x=0, \ \ & x=0,L,\\
		v_x(0,t)=0,~v(L,t)=b,&
\end{cases}\end{equation}
where $b>0$ is a constant. In other words, we prescribe the no-flux boundary condition for the bacteria and mixed nonhomogeneous Dirichlet-Neumann boundary condition for the oxygen. Lauffenburger \emph{et al.}~\cite{LAK} imposed the same mixed conditions to demonstrate the effects of biophysical transport processes along with biochemical reaction processes upon steady state bacterial population growth. This type of boundary conditions were also adopted by Tuval \emph{et al.} \cite{Tuval} to simulate the accumulation layers formed by aerobic bacteria in their experiment.

In this paper, we first show that the system \eqref{original problem}-\eqref{original boundary condition} has a unique nonconstant steady state $(U,V)$, where $U$ approaches a transition layer as the chemotactic coefficient $\chi\rightarrow\infty$. Then we show that the steady state $(U,V)$ is nonlinearly stable if the initial value $(u_0,v_0)$ is an appropriate small perturbation of $(U,V)$.

To show the existence of steady state, we use a parametric method along with the continuity argument inspired by the work of \cite{Lee}. By observing a monotonicity structure admitted by the volume-filling effect, we obtain the uniqueness of steady state based on the consumption mechanism of the model. It is a little bit tricky to derive the asymptotic profile of $(U,V)$ as $\chi\rightarrow\infty$. We first observe that both $U$ and $V$ are monotone, and $\int_0^LU(x)dx=m$ with $m$ being a constant independent of $\chi$. As in \cite{Wang2000, Wang-Xu13}, where the Helly's compactness theorem is applied, $U$ has a strong limit $U_\infty$ in $L^1(0,L)$. On the other hand, using the consumption mechanism of $V$, it is easy to establish the uniform estimates of $V$ in $W^{2,\infty}\left( 0,L\right)$. Hence, $V$ has a limit $V_\infty$ in $W^{1,\infty}\left( 0,L\right)$. After passing to the limit, we obtain the coupled equations of $(U_\infty,V_\infty)$. By frequently using the monotonicity of $(U_\infty,V_\infty)$, and the fact that $\int_0^LU_\infty(x)dx=m$, we then characterize that $U_\infty$ is a step function. Since $V_\infty$ is $W^{2,\infty}$ smooth, one can directly obtain its formula using its equation and the formula of $U_\infty$. To show the stability of steady state (with fixed $\chi$), we first notice that the bacterial mass is conserved, which stimulates us to take the anti-derivative method to reformulate the problem into new unknowns. Then thanks to a cancellation structure of the perturbation equations, by exploiting the dissipative mechanism of the model, along with the technique of \emph{a priori} assumption, we obtain the stability of steady state without any assumption on the parameters.

The remainder of this paper is organized as follows. In Section \ref{main results}, we state our main results on
the existence, uniqueness and asymptotic profile of steady states (Theorem \ref{existence theorem}) and nonlinear stability of steady states (Theorem \ref{stability theorem}). In Section \ref{steady states}, we study the stationary problem and prove Theorem \ref{existence theorem}. Section \ref{nonlear stability} is devoted to the proof of Theorem \ref{stability theorem}. In Section \ref{numerical}, we carry out some numerical simulations for the system  to illustrate and confirm our nolinear stability results of steady states.

\section{Statement of main results}\label{main results}

In this section, we shall initially present the results concerning the existence, uniqueness and asymptotic profile of steady states and then state our main results on the nonlinear stability of steady states. Throughout the paper, we denote by $H^k$ the standard Sobolev space $H^k(0,L)$. The integrals $\int_0^Lf(x)\mathrm{d}x$ and $\int_{0}^{t}\mathrm{e}^{\alpha \tau}\int_0^Lf(x,\tau)\mathrm{d}x\mathrm{d}\tau$ will be abbreviated as $\int f(x)$ and $\int_{0}^{t}\mathrm{e}^{\alpha \tau}\int f(x,\tau)$, respectively. For any function $g(x)$, we denote $g^\prime(x)=\frac{\mathrm{d}g}{\mathrm{d}x}$.

In view of the no-flux boundary condition of $u$, one can easily see that the mass of bacteria is conserved, that is
\begin{equation}\label{m}
	\int_0^Lu(x,t)dx=\int_0^Lu_0(x)dx\triangleq m, \ \forall t>0.
\end{equation}
Here we denote by $m$ the bacterial mass. Since it is expected that $u(x,t)\rightarrow U(x)$ as $t\rightarrow\infty$, $U$ should also satisfy $\int_0^LU(x)dx=m$. Therefore, the steady state of  the system \eqref{original problem}-\eqref{original boundary condition} satisfies
\begin{equation}\label{steady state problem}
	\begin{cases}{\left[\left(q(U)-q^{\prime}(U) U\right) U^\prime-\chi q(U) U V^\prime\right]^\prime=0,} & x \in(0,L), \\ V^{\prime\prime}=U V, & x \in(0,L), \\
		\int_0^L U(x) d x=m, \end{cases}
\end{equation}
with boundary conditions	
\begin{equation}\label{steady state boundary conditon}
	\left\{\begin{array}{l}
		\left(q(U)-q^{\prime}(U) U\right) U^\prime-\chi q(U) U V^\prime=0 \ \ \text{ at } x=0,L, \\
		V^\prime(0)=0,~ V(L)=b.
	\end{array}\right.
\end{equation}	
We also need the following assumptions on $q$:
\begin{itemize}
	\item[\textbf{(A)}] $q(u)$ is smooth, and there exists a constant $K>0$ such that $q(u)\equiv0$ for $u\geq K$ and $q(u)>0$, $q'(u)<0$, $q(u)$ is smooth on $(0,K)$.
\end{itemize}
A prototypical choice of $q(u)$ is $q(u)=\left[\left(1-\frac{u}{K}\right)^+ \right]^\gamma$ for $\gamma>0$, where $f^+=\max\{f,0\}$.
\begin{remark}
	Assumption \textbf{(A)} is proposed based on specific biological considerations. Recalling that $q(u)$ represents the probability of cells finding space in neighboring locations, when the cell capacity is not yet saturated ($0<u<K$), cells infiltrate, leading to $q(u)>0$. As the cell population grows, the probability of movement decreases, resulting in $q'(u)<0$ for all $0<u<K$. At full saturation ($u\geq K$), no further cells could fill in, so $q(u)=0$ for all $u\geq K$.
\end{remark}

Our first result on the existence, uniqueness and asymptotic profile of steady states is given below.
\begin{theorem}\label{existence theorem}  Assume \textbf{(A)} holds. Then for any $m\in\left(0,KL\right)$, the problem \eqref{steady state problem}-\eqref{steady state boundary conditon} admits a unique classical non-constant solution $(U,V)$ satisfying
	\begin{equation}\label{UV}
		0<U(x)<K, \quad 0<V(x) \leq b \quad \forall x \in [0,L].
	\end{equation}
	Furthermore, as $\chi \rightarrow \infty$,
	\begin{equation}
		U \rightarrow U_{\infty}=\begin{cases}
			0,~&x \in\left[0,L-\frac{m}{K}\right), \\
			K,~&x \in\left(L-\frac{m}{K}, L\right],
		\end{cases}
	\end{equation} pointwise; and
	\begin{equation}
		V \rightarrow V_{\infty}=\begin{cases}
			\frac{2b\mathrm{e}^{\frac{m}{\sqrt{K}}}}{1+\mathrm{e}^{\frac{2m}{\sqrt{K}}}},~&x \in\left[0,L-\frac{m}{K}\right], \\
			C_1 \mathrm{e}^{-\sqrt{K} x}+C_2 \mathrm{e}^{\sqrt{K} x},~&x \in\left[L-\frac{m}{K}, L\right],
		\end{cases}	
	\end{equation}uniformly  on $[0,L]$,
	where $C_1=b\mathrm{e}^{\sqrt{K}L}/(1+\mathrm{e}^{\frac{2m}{\sqrt{K}}})$ and $C_2=b\mathrm{e}^{-\sqrt{K}L+\frac{2m}{\sqrt{K}}}/(1+\mathrm{e}^{\frac{2m}{\sqrt{K}}})$.
\end{theorem}

Over the past few decades, a broad mathematical literature was devoted to the classical chemotaxis-consumption model in which $q\equiv1$ of \eqref{original problem}. Significant progress has been made in understanding the global existence and long-time behavior of solutions (cf. \cite{JinHY,Lankeit-Long-time,one-dimension-S,Lankeit-2017,Li-Li,TaoY2011,TaoY2012}). However, much less is known about the chemotaxis-consumption model with volume-filling effect. Our second result is the establishment of the global existence of solutions to \eqref{original problem} expressed through a nonlinear squeezing probability and the proof that, as time goes to infty, these solutions stabilize towards steady states obtained in Theorem \ref{existence theorem}.

\begin{theorem}\label{stability theorem} Let $\chi>0$ be fixed. Assume \textbf{(A)} holds. Suppose that $(u_0,v_0)\in H^1$ with $0\leq u_0(x)<K$, $v_0(x)\geq0$ for all $x\in[0,L]$, and that $v_0(L)=b$. Let $\left(U,V\right)$ be the steady state obtained in Theorem \ref{existence theorem} with $m=\int_0^L u_0 d x$. Define
	\[\phi_0(x)=\int_0^x\left(u_0(y)-U(y)\right) \mathrm{d} y.\]
	Then there exists a constant $\epsilon$ such that if $\left\|\phi_0\right\|_{H^2}^2+\left\|v_0-V\right\|_{H^1}^2\leq \epsilon$,	
	then the initial boundary value problem \eqref{original problem}-\eqref{original boundary condition} admits a unique global solution $(u,v)$ satisfying
	$$
	u -U\in C\left([0, \infty) ; H^1\right) \cap L^2\left(0, \infty ; H^2\right), \quad v-V \in C\left([0, \infty) ; H^1\right) \cap L^2\left(0, \infty ; H^2\right).
	$$
	Furthermore, the solution has the following asymptotic decay
	\begin{equation}\label{convergence}
		\|(u-U, v-V)(\cdot, t)\|_{H^1} \leq C\left\|(u_0-U,v_0-V)\right\|_{H^1}\mathrm{e}^{-\alpha_0 t} \text { for any } t \geq 0,	
	\end{equation}
	where $\alpha_0$ and $C$ are positive constants independent of $t$.
\end{theorem}

	\begin{remark} We need $0\leq u_0(x)<K$ for all $x\in[0,L]$  to simplify the dynamics of the system  \eqref{original problem}-\eqref{original boundary condition}. Otherwise, if $u_0(x)=K$ for some $x\in [0,L]$, the system  \eqref{original problem}-\eqref{original boundary condition} might be degenerate since either the diffusion coefficient of $u$ or the chemotactic coefficient becomes zero at the point where $u(x,t)=K$. Furthermore, if $u_0(x)>K$ for some $x\in [0,L]$, one may encounter a free boundary problem.
	Both the degenerate problem and the free boundary problem are quite challenging, and we leave the problems of dynamics of the system  \eqref{original problem}-\eqref{original boundary condition} with general initial conditions for future study.
	
	It is an interesting question to investigate the effect of $\chi$ on the large time behavior of solutions. However, in the current argument, the background solutions $(U,V)$  depend on $\chi$, and particularly, the limit of bacterial density as $\chi$ goes to infinity is a step function. This makes it challenging to derive an explicit formula for the effect of large $\chi$ on the dynamical stability. Our current result is a first attempt to explore the formation of bacterial aggregation for the chemotaxis-consumption model under volume-filling effect. We believe that, as in the chemotaxis-production model, large $\chi$ also affects the dynamics of the chemotaxis-consumption model. We leave this interesting problem for future study.

\end{remark}

\section{Steady states}\label{steady states}

In this section, we shall study the stationary problem \eqref{steady state problem}-\eqref{steady state boundary conditon} and prove Theorem \ref{existence theorem}.
\begin{proof}[Proof of Theorem \ref{existence theorem}--existence] We shall follow the framework of \cite{Lee} to show the existence of $(U,V)$.
	Owing to the no-flux boundary condition of $U$, it is easy to see that the steady state $(U,V)$ satisfies
	\begin{equation}\label{transformed steady state problem}
		\begin{cases}\left(q(U)-q^{\prime}(U) U\right) U^{\prime}-\chi q(U) U V^{\prime}=0,  \\ V^{\prime\prime}=U V,\\V^{\prime}(0)=0,~ V(L)=b,\\
			\int_0^L U(x) d x=m.\end{cases}
	\end{equation}	We attempt to find the solution $(U,V)$ satisfying $V(x)>0$ and $0<U(x)<K$ ensuring that $q(U(x))>0$ for all $x\in[0,L]$.
	The first equation of \eqref{transformed steady state problem} gives
	\begin{equation}\label{eq25}
		\left( \ln \frac{U}{q(U)}\right)^{\prime}(x)=\chi V^{\prime}(x),
	\end{equation}
	which further leads to
	\begin{equation}\label{3.3}
		\frac{U}{q(U)}=\lambda \mathrm{e}^{\chi V} \text{ for some constant }\lambda>0.
	\end{equation}
	We next divide our argument into three steps.

	\textit{Step 1.} Set $G(U)\triangleq\frac{U}{q(U)}$, then
	\begin{equation}\nonumber
		G'(U)=\frac{q(U)-q'(U)U}{q^2(U)}>0 \text{ for } U\in(0,K),	
	\end{equation}
	due to $q(U)>0$ and $q'(U)<0$ for $U\in(0,K)$. Thus $U=G^{-1}\left(\lambda \mathrm{e}^{\chi V} \right)\in(0,K)$. Substituting this formula into the second equation of \eqref{transformed steady state problem}, we have
	\begin{equation}\label{Vxx=UV}
		\begin{cases}
			V^{\prime\prime}=G^{-1}\left(\lambda \mathrm{e}^{\chi V}\right)V,\\
			V^{\prime}(0)=0,~V(L)=b,\\
			\int_{0}^{L}G^{-1}\left(\lambda \mathrm{e}^{\chi V} \right)\mathrm{d}x=m.
		\end{cases}
	\end{equation}	
	It suffices to construct a solution pair $(\lambda,V)$ of \eqref{Vxx=UV} with $\lambda>0$ and $0<V\leq b$. To achieve this, we first fix $\lambda>0$ and construct a solution $V(x;\lambda)$ of
	\begin{equation}\label{3.5}
		\begin{cases}
			V^{\prime\prime}=G^{-1}\left(\lambda \mathrm{e}^{\chi V}\right)V,\\
			V^\prime(0)=0,~V(L)=b.
		\end{cases}
	\end{equation}	
	It is apparently that the constant $b$ serves as an upper solution of \eqref{3.5}, while $0$ is a lower solution. By employing the standard monotone iteration scheme and noting that the function $f(s)=G^{-1}\left(\lambda \mathrm{e}^{\chi s}\right)s$ is increasing for $s>0$, we obtain that \eqref{3.5} admits a unique classical solution $V(x;\lambda)$ satisfying
	\begin{equation*}
		0<V(x;\lambda)\leq b \quad \text{on }[0,L].
	\end{equation*}

	\textit{Step 2.} We next show the continuity of $V(x;\lambda)$ with respect to $\lambda$ in some topology. Let $\lambda_1>\lambda_2>0$ and $V_{\lambda_i}$ be the solutions of \eqref{3.5} with $\lambda=\lambda_i$, $i=1,2$, respectively. One may check that $V_{\lambda_1}-V_{\lambda_2}$ satisfies
	\begin{equation}\label{V1-V2}
		\begin{cases}
			-\left(V_{\lambda_{1}}-V_{\lambda_{2}} \right)^{\prime\prime}+G^{-1}\left(\lambda_2 \mathrm{e}^{\chi V_{\lambda_2}}\right)\left(V_{\lambda_{1}}-V_{\lambda_{2}}\right)=-F\left(V_{\lambda_{1}}, V_{\lambda_{2}}\right)V_{\lambda_{1}},\\
			\left(V_{\lambda_{1}}-V_{\lambda_{2}} \right)^\prime(0)=0,~\left(V_{\lambda_{1}}-V_{\lambda_{2}} \right)(L)=0,
		\end{cases}
	\end{equation}	
	where
	$F\left(V_{\lambda_{1}}, V_{\lambda_{2}}\right)\triangleq G^{-1}\left(\lambda_1 \mathrm{e}^{\chi V_{\lambda_1}}\right)-G^{-1}\left(\lambda_2 \mathrm{e}^{\chi V_{\lambda_2}}\right)$.	
	A direct calculation gives
	\begin{equation}\label{F}
		\begin{aligned}
			&\quad F\left(V_{\lambda_{1}}, V_{\lambda_{2}}\right)(x)\\&=\int_{0}^{1}\frac{\mathrm{d}G^{-1}}{\mathrm{d}s}\left(s\lambda_{1}\mathrm{e}^{\chi V_{\lambda_1}}+\left(1-s\right)\lambda_{2}\mathrm{e}^{\chi V_{\lambda_2}}\right) \mathrm{d}s	\\&= \int_{0}^{1}\left(G^{-1} \right)'\left(s\lambda_{1}\mathrm{e}^{\chi V_{\lambda_1}}+\left(1-s\right)\lambda_{2}\mathrm{e}^{\chi V_{\lambda_2}}\right) \mathrm{d}s	\cdot\left(\lambda_{1}\mathrm{e}^{\chi V_{\lambda_1}}-\lambda_{2}\mathrm{e}^{\chi V_{\lambda_2}}\right)\\&=\int_{0}^{1}\left(G^{-1} \right)'\left(s\lambda_{1}\mathrm{e}^{\chi V_{\lambda_1}}+\left(1-s\right)\lambda_{2}\mathrm{e}^{\chi V_{\lambda_2}}\right) \mathrm{d}s	\cdot\lambda_{1}\left(\mathrm{e}^{\chi V_{\lambda_1}}-\mathrm{e}^{\chi V_{\lambda_2}}\right)\\&\quad+\int_{0}^{1}\left(G^{-1} \right)'\left(s\lambda_{1}\mathrm{e}^{\chi V_{\lambda_1}}+\left(1-s\right)\lambda_{2}\mathrm{e}^{\chi V_{\lambda_2}}\right) \mathrm{d}s	\cdot\mathrm{e}^{\chi V_{\lambda_2}}\left(\lambda_{1}-\lambda_{2}\right)\\&=\xi\left(V_{\lambda_{1}}, V_{\lambda_{2}}\right)\cdot\lambda_{1}\cdot \eta\left(V_{\lambda_{1}}, V_{\lambda_{2}}\right)	\cdot\chi\left(V_{\lambda_{1}}-V_{\lambda_{2}}\right)+\xi\left(V_{\lambda_{1}}, V_{\lambda_{2}}\right)	\cdot\mathrm{e}^{\chi V_{\lambda_2}}\left(\lambda_{1}-\lambda_{2}\right),
		\end{aligned}
	\end{equation}
	where
	$$\xi\left(V_{\lambda_{1}}, V_{\lambda_{2}}\right)(x)\triangleq\int_{0}^{1}\left(G^{-1} \right)'\left(s\lambda_{1}\mathrm{e}^{\chi V_{\lambda_1}}+\left(1-s\right)\lambda_{2}\mathrm{e}^{\chi V_{\lambda_2}}\right) \mathrm{d}s>0,$$
	and
	$$
	\eta\left(V_{\lambda_{1}}, V_{\lambda_{2}}\right)(x)\triangleq\int_{0}^{1}\mathrm{e}^{s\chi V_{\lambda_1}+(1-s)\chi V_{\lambda_2}}\mathrm{d}s>0.
	$$
	Multiplying \eqref{V1-V2} by $V_{\lambda_{1}}-V_{\lambda_{2}}$, integrating the equation over $(0,L)$ and using
	\eqref{F}, we have
	\begin{equation}\label{3.8}
		\begin{aligned}
			&\quad\int_0^L\left|\left(V_{\lambda_{1}}-V_{\lambda_{2}} \right)^\prime\right|^2+\int_0^L G^{-1}\left(\lambda_2 \mathrm{e}^{\chi V_{\lambda_2}}\right)\left(V_{\lambda_{1}}-V_{\lambda_{2}}\right)^2\\&\quad+\lambda_{1}\chi\int_0^L \xi\left(V_{\lambda_{1}}, V_{\lambda_{2}}\right) \eta\left(V_{\lambda_{1}}, V_{\lambda_{2}}\right)V_{\lambda_{1}}\left(V_{\lambda_{1}}-V_{\lambda_{2}}\right)^2\\&=\left(\lambda_{1}-\lambda_{2}\right)\int_0^L \xi\left(V_{\lambda_{1}}, V_{\lambda_{2}}\right)	V_{\lambda_{1}}\mathrm{e}^{\chi V_{\lambda_2}}\left(V_{\lambda_{1}}-V_{\lambda_{2}}\right).
		\end{aligned}
	\end{equation}
	Thanks to the Sobolev inequality $\|V_{\lambda_{1}}-V_{\lambda_{2}}\|_{L^\infty}\leq L^{\frac{1}{2}} \|\left(V_{\lambda_{1}}-V_{\lambda_{2}} \right)^\prime\|_{L^2}$, we get from \eqref{3.8} that
	\begin{equation*}
		\|V_{\lambda_{1}}-V_{\lambda_{2}}\|^2_{L^\infty}\leq \left(\lambda_{1}-\lambda_{2}\right)\cdot L\int_0^L \xi\left(V_{\lambda_{1}}, V_{\lambda_{2}}\right)	V_{\lambda_{1}}\mathrm{e}^{\chi V_{\lambda_2}}\left(V_{\lambda_{1}}-V_{\lambda_{2}}\right).
	\end{equation*} It then follows from H\"{o}lder's inequality that
	\begin{equation*}
		\begin{split}
			\|V_{\lambda_{1}}-V_{\lambda_{2}}\|_{L^\infty}&\leq |\lambda_{1}-\lambda_{2}|\cdot L\int_0^L \xi\left(V_{\lambda_{1}}, V_{\lambda_{2}}\right)	V_{\lambda_{1}}\mathrm{e}^{\chi V_{\lambda_2}}\\
			&\leq|\lambda_{1}-\lambda_{2}|\cdot Lb\mathrm{e}^{\chi b}\int_0^L \xi\left(V_{\lambda_{1}}, V_{\lambda_{2}}\right).
	\end{split}\end{equation*}
	Thus, $V(x;\lambda)$ is continuous with respect to $\lambda$.

	\textit{Step 3.} Noting $0<V(x;\lambda)\leq b$, it holds $\lambda<\lambda\mathrm{e}^{\chi V}\leq\lambda\mathrm{e}^{\chi b}$. Recalling that the function $G(s)$ is defined by $G(s)=\dfrac{s}{q(s)}$, we have
	\[\lim _{\lambda \rightarrow 0^{+}}G^{-1}\left( \lambda\mathrm{e}^{\chi V_\lambda}\right)=G^{-1}(0)=0\quad \text { and }  \quad\lim _{\lambda \rightarrow \infty}G^{-1}\left( \lambda\mathrm{e}^{\chi V_\lambda}\right) =K. \]Thus,
	\begin{equation}\nonumber
		\lim _{\lambda \rightarrow 0^{+}}  \int_{0}^{L}G^{-1}\left( \lambda\mathrm{e}^{\chi V}\right)  \mathrm{d} x=0 \quad \text { and } \quad \lim _{\lambda \rightarrow \infty} \int_{0}^{L}G^{-1}\left( \lambda\mathrm{e}^{\chi V}\right) \mathrm{d} x=KL>m.
	\end{equation}
	Then by continuity of $V(x;\lambda)$\ in $\lambda$ established in Step 2, we can find a constant $\lambda_m>0$  such that the corresponding solution $V(x;\lambda_m)$ of \eqref{3.5} satisfies $\int_0^LG^{-1}\left(\lambda_m \mathrm{e}^{\chi V(x;\lambda_m)} \right)\mathrm{d}x=m$. Consequently, $V(x;\lambda_m)$ is a solution of \eqref{Vxx=UV}. Define $U(x;\lambda_m)\triangleq G^{-1}\left(\lambda_m \mathrm{e}^{\chi V(x;\lambda_m)} \right)$. Then $(U(x;\lambda_m),V(x;\lambda_m))$ is a solution of \eqref{transformed steady state problem}. We thus finish the proof of the existence of steady states.
\end{proof}

\begin{proof}[Proof of Theorem \ref{existence theorem}--uniqueness]	
	Suppose that there are two solutions $(U_1,V_1)$ and $(U_2,V_2)$ of \eqref{transformed steady state problem} satisfying \eqref{UV}. Then by \eqref{3.3}, there are two positive real numbers $\lambda_1$ and $\lambda_2$ such that
	\begin{equation}\nonumber
		\frac{U_1}{q(U_1)}=\lambda_1\mathrm{e}^{\chi V_1},\quad	 \frac{U_2}{q(U_2)}=\lambda_2\mathrm{e}^{\chi V_2}.
	\end{equation}
	We assume without loss of generality that $0<\lambda_1\leq \lambda_2$. If $\lambda_1=\lambda_2$, then $V_1$ and $V_2$ are two solutions of \eqref{3.5} with the same $\lambda$. According to the uniqueness of solution to \eqref{3.5}, we have $V_1=V_2$ and hence $U_1=U_2$.
	
	If $0<\lambda_1<\lambda_2$, since $V_{1}$ satisfies $-\left(V_{1} \right)^{\prime\prime}+G^{-1}\left(\lambda_1 \mathrm{e}^{\chi V_{1}}\right)V_{_1}=0$, one immediately obtains from the monotonicity of $G^{-1}(s)$ for $s>0$ that $-\left(V_{1} \right)^{\prime\prime}+G^{-1}\left(\lambda_2 \mathrm{e}^{\chi V_{1}}\right)V_{1}>0$, which along with the standard comparison theorem for \eqref{3.5} yields
	\begin{equation}\label{V_1>V_2}
		V_1(x)>V_2(x),\quad	x\in (0,L).
	\end{equation}
	On the other hand, at $x=L$, we get
	\begin{equation}\label{eq24}
		\frac{U_1(L)}{q(U_1(L))}=\lambda_1\mathrm{e}^{\chi b}<\lambda_2\mathrm{e}^{\chi b}=\frac{U_2(L)}{q(U_2(L))},\end{equation}
	which in combination with $\frac{\mathrm{d}}{\mathrm{d}U}(\frac{U}{q(U)})>0$ gives $U_1(L)<U_2(L)$.
	
	For convenience we set $\frac{U(x)}{q(U(x))}\triangleq W(x)$ in the rest of the proof. Then, by virtue of \eqref{eq25} and \eqref{eq24}, we find $W$ satisfies
	\begin{equation}\label{W euq}
		W^\prime=\chi W V^\prime \ \text{on } (0,L),	
	\end{equation}
	and
	\begin{equation}\label{eq26}
		W_1(L)<W_2(L).
	\end{equation}	
	There are three cases regarding the profile of $W_1$ and $W_2$.
	
	\textit{Case 1.} If $W_1(x)\leq W_2(x)$ for all $x\in (0,L)$ and $W_1(x)\not\equiv W_2(x)$, according to $\frac{\mathrm{d}W}{\mathrm{d}U}=\frac{\mathrm{d}}{\mathrm{d}U}(\frac{U}{q(U)})>0$, we have $U_1(x)\leq U_2(x)$ for all $x\in(0,L)$ and $U_1(x)\not\equiv U_2(x)$, which contradicts $\int_{0}^{L}U_1(x)\mathrm{d} x=\int_{0}^{L}U_2(x)\mathrm{d} x=m$.
	
	\textit{Case 2.} If $W_1(x)-W_2(x)$ changes sign only once, let's assume without loss of generality that it changes sign at a point $x_0\in(0,L)$. Then, we have $W_1(x_0)=W_2(x_0)$ and
	\begin{equation}\label{eq27}
		W_1'(x_0)\leq W_2'(x_0),
	\end{equation}
	due to \eqref{eq26}. Since $W_1(x)-W_2(x)$ changes sign only once, we get $W_1(x)>W_2(x)$ on $(0,x_0)$, and hence
	\begin{equation}\label{eq28}
		U_1(x)>U_2(x)\quad \text{on } (0,x_0),	
	\end{equation}
	because $W(U)$ is monotonically increasing in $U$. On the other hand, in view of \eqref{W euq} and the second equation of \eqref{transformed steady state problem}, one obtains
	\begin{equation}\nonumber
		\begin{aligned}
			W_1^{\prime}(x_0)&=\chi W_1(x_0)V_1'(x_0)=\chi W_2(x_0)V_1'(x_0)=\chi W_2(x_0)\int_{0}^{x_0}U_1 V_1(s)\mathrm{d}s\\&>\chi W_2(x_0)\int_{0}^{x_0}U_2 V_2(s)\mathrm{d}s =\chi W_2(x_0)V_2'(x_0)=W_2^{\prime}(x_0),	\end{aligned}
	\end{equation}
	where we have used \eqref{V_1>V_2} and \eqref{eq28} for the inequality. This contradicts to \eqref{eq27}.
	
	\textit{Case 3.} If $W_1(x)-W_2(x)$ changes sign at least twice, we take $x_0$ and $x_1$ with $x_0< x_1$ as the last two points where $(W_1-W_2)(x)$ changes sign. Then, we have
	\begin{equation}\label{eq29}
		W_1(x_0)=W_2(x_0),\quad W_1(x_1)=W_2(x_1),
	\end{equation}
	and
	\begin{equation}\label{eq30}
		W_1'(x_0)\geq W_2'(x_0),\quad W_1'(x_1)\leq W_2'(x_1),
	\end{equation}where we have used \eqref{eq26} in \eqref{eq30}.
	As demonstrated in the proof of \eqref{eq28}, it holds that
	\begin{equation}\label{eq31}
		U_1(x)>U_2(x)\quad \text{on } (x_0,x_1).	
	\end{equation}
	Recalling that $W^\prime=\chi W V^\prime$, combining \eqref{eq29} and \eqref{eq30}, we deduce that
	\begin{equation}\nonumber
		V_1'(x_0)\geq V_2'(x_0),	
	\end{equation}
	which along with the second equation of \eqref{transformed steady state problem}, \eqref{V_1>V_2} and \eqref{eq31}, gives rise to
	\begin{equation}\label{eq32}
		V_1'(x_1)=V_1'(x_0)+\int_{x_0}^{x_1}U_1 V_1(s)\mathrm{d}s>V_2'(x_0)+\int_{x_0}^{x_1}U_2V_2(s)\mathrm{d}s=V_2'(x_1).	
	\end{equation}
	By \eqref{eq29} and \eqref{eq32}, we thus get
	\begin{equation}
		W_1'(x_1)=\chi W_1(x_1)	V_1'(x_1)= \chi W_2(x_1)V_1'(x_1)>\chi W_2(x_1)V_2'(x_1)=W_2'(x_1),
	\end{equation}
	which contradicts to the second inequality of \eqref{eq30}.
	
	All of the three cases show the contradictions, which implies $\lambda_1<\lambda_2$ is not true, leading to the conclusion that only $\lambda_1=\lambda_2$ holds. We thus obtain the uniqueness of the steady state.
\end{proof}

In order to study the asymptotic profile of the steady states, we present some elementary estimates for $(U,V)$.

\begin{lemma}
	Let $(U,V)$ be the steady state solution of \eqref{original problem}-\eqref{original boundary condition} stated in Theorem \ref{existence theorem}. Then it holds that
	\begin{equation}\label{U_x,V_x>0}
		V^\prime\geq0,~ U^\prime\geq0,\quad x\in [0,L],
	\end{equation}
	\begin{equation}\label{inequality}
		\left(V^\prime \right)^2\leq UV^2,\quad x\in [0,L],
	\end{equation}
	\begin{equation}\label{V_xU_xV_xx,U_xx}
		|V^\prime|\leq C,~|V^{\prime\prime}|\leq C,\quad x\in [0,L],
	\end{equation}where $C>0$ is a constant independent of $\chi$.
\end{lemma}

\begin{proof}
	Since $(U,V)(x)>0$ for all $x\in[0,L]$, by the second equation of \eqref{transformed steady state problem}, we get $V^{\prime\prime}(x)>0$ for all $x\in[0,L]$, which implies $V^\prime(x)>0$ since $V^\prime(0)=0$. Noting $q(U)>0$ and $q'(U)<0$ for $U\in(0,K)$, the first equation of \eqref{transformed steady state problem} gives
	\begin{equation*}
		U^\prime=\frac{\chi q(U)U}{q(U)-q'(U)U}V^\prime\geq0 \text{ for all } x\in[0,L].
	\end{equation*}	
	Multiplying the second equation of \eqref{steady state problem} by $V^\prime$ and integrating over $(0,x)$, we have
	\begin{equation}\nonumber
		\frac{\left(V^\prime \right)^2(x)}{2}=\int_0^x UVV^\prime\mathrm{d}y=\int_{V(0)}^{V(x)}G^{-1}\left(\lambda\mathrm{e}^{\chi s}\right)s\mathrm{d}s\leq \int_{0}^{V(x)}G^{-1}\left(\lambda\mathrm{e}^{\chi s}\right)s\mathrm{d}s.
	\end{equation}
	Hence, we get, thanks to $V>0$ and $G'>0$  that
	\begin{equation}\nonumber
		\frac{\left(V^\prime \right)^2}{2}\leq \int_{0}^{V}G^{-1}\left(\lambda\mathrm{e}^{\chi s}\right)s\mathrm{d}s=\frac{G^{-1}\left(\lambda\mathrm{e}^{\chi V}\right)V^2 }{2}-\frac{1}{2}\int_0^V\frac{\lambda\chi\mathrm{e}^{\chi s}}{G'\left(\lambda\mathrm{e}^{\chi s}\right)}s^2 \mathrm{d}s\leq \frac{G^{-1}\left(\lambda\mathrm{e}^{\chi V}\right)V^2 }{2}= \frac{UV^2 }{2}.
	\end{equation}
	Based on this inequality, it holds
	\[|V^\prime(x)|\leq|U|^{\frac{1}{2}}V(x)\leq\sqrt{K}b, \ x\in[0,L].\]
	The second equation of \eqref{transformed steady state problem} leads to
	$$|V^{\prime\prime}(x)|=UV(x)\leq Kb, \ x\in[0,L].$$ The proof is completed.
\end{proof}

\begin{proof}[Proof of Theorem \ref{existence theorem}--asymptotic profile of $(U,V)$] Since $0<U(x)<K$  for all $x\in[0,L]$ and $U$ is monotonically increasing on $[0,L]$, by Helly's compactness theorem, there exists $U_\infty\in L^\infty(0,L)$ such that $$U\rightarrow U_\infty \text{ pointwisely on } [0,L] \text{ as } \chi\rightarrow\infty.$$ By Lebesgue Dominated Convergence Theorem, we have $\int_{0}^{L} U_{\infty}(x)\mathrm{d}x=m$. Moreover,
	$$0\leq U_\infty(x)\leq K \text{ on }(0,L).$$
	
	By \eqref{UV} and \eqref{V_xU_xV_xx,U_xx}, there exists a constant $M$ independent of $\chi$ such that $\|V\|_{C^2[0,L]}\leq M$. According to Arzela-Ascoli theorem, there exists $V_{\infty}\in  W^{2,\infty}\left( 0,L\right)$ such that $$V\rightarrow V_{\infty} \text{ strongly in } W^{1,\infty}\left( 0,L\right) \text{ as } \chi\rightarrow\infty.$$
	
	We next characterize the formula of $(U_\infty,V_\infty)$. Dividing the first equation of \eqref{transformed steady state problem} by $\chi$, it holds that
	\begin{equation}\label{eq33}
		\frac{1}{\chi}\left(q(U)-q^{\prime}(U) U\right) U^\prime= q(U) U V^\prime.	
	\end{equation}
	Denote by $g(U)\triangleq\int_0^Uq(s)s(q(s)-q'(s)s)ds$, then since $q(s)$ is smooth on $[0,K]$, it is easy to see that $|g(U)|\leq C$ for some constant $C$ independent of $\chi$. Now multiplying \eqref{eq33} by $(q(U)-q^{\prime}(U) U)U^\prime$ and integrating the equation on $(0,L)$, we get
	\begin{equation*}\nonumber
		\begin{split}
			\frac{1}{\chi}\int_0^L \left(q(U)-q^{\prime}(U) U\right)^2 \left(U^\prime \right)^2&= \int_0^L q(U) U\left(q(U)-q^{\prime}(U) U\right) U^\prime V^\prime=\int_0^L g(U)_xV^\prime\\&=g(U)V^\prime\Big|_0^L-\int_0^L g(U)V^{\prime\prime}\\&=g(U)V^\prime(L)-\int_0^L g(U)UV\\&\leq C,\end{split}
	\end{equation*} where we have used \eqref{V_xU_xV_xx,U_xx} and \eqref{UV} in the last inequality.
	Thus,
	\begin{equation}\label{3.25}
		\begin{aligned}
			\int_0^L \frac{ 1}{\chi^2}\left(q(U)-q^{\prime}(U) U\right)^2 \left( U^\prime\right)^2&=\frac{1}{\chi}\cdot\frac{1}{\chi}\int_0^L\left(q(U)-q^{\prime}(U) U\right)^2\left( U^\prime\right)^2\\&\leq \frac{C}{\chi}\rightarrow 0\quad \text{as } \chi\rightarrow\infty.
	\end{aligned}\end{equation}
	Multiplying \eqref{eq33} by $\omega\in C_0^\infty(0,L)$, we get
	\begin{equation}\label{3.26}
		\int_0^L\frac{ 1}{\chi}\left(q(U)-q^{\prime}(U) U\right) U^\prime \omega =\int_0^Lq(U) U V^\prime\omega.
	\end{equation} Letting $\chi\rightarrow\infty$ in \eqref{3.26}, by \eqref{3.25}, we have
	\begin{equation*}
		\int_0^Lq(U_\infty) U_\infty V_\infty^\prime \omega=0,\quad \forall \omega \in C_0^\infty(0,L).
	\end{equation*} Thus,
	\begin{equation}\label{qUinftyUinfty Vinfty x}
		q(U_\infty)	U_\infty V_\infty^\prime=0 \text{ on } (0,L).
	\end{equation}
	Applying a similar argument on the second equation of \eqref{transformed steady state problem} leads to
	\begin{equation}\label{Vinfty xxUinfty Vinfty}
	V_\infty^{\prime\prime}=U_\infty V_\infty \text{ on } (0,L).
	\end{equation}

	We next show that
	\begin{equation}\label{Uinfty}
		U_{\infty}(x)=\begin{cases}
			0,~&x \in\left[0,L-\frac{m}{K}\right), \\
			K,~&x \in\left(L-\frac{m}{K}, L\right].
		\end{cases}
	\end{equation}	We prove \eqref{Uinfty} by a contradiction argument.
	If there exists	an $x_1\in \left(L-\frac{m}{K}, L\right]$ such that $U_\infty(x_1)<K$, then by the monotonicity of $U_\infty(x)$, we get $U_\infty(x)<K$ and $q(U_\infty(x))>0$ for any $x \in [0,x_1]$. Therefore, due to \eqref{qUinftyUinfty Vinfty x}, it gives
	\begin{equation}\label{Uinfty Vinfty x}
		U_\infty V_\infty^\prime(x)=0\quad \text{on }[0,x_1].
	\end{equation}
	We next show that $U_\infty(x)\equiv0$ on $[0,x_1]$. Otherwise, there exists an $x_2\in [0,x_1]$ such that $U_\infty(x_2)>0$. Multiplying \eqref{Vinfty xxUinfty Vinfty} by $V_\infty^\prime$, we then get
	\begin{equation}
		\frac{\left(V^\prime_{\infty }\right) ^2(x)}{2}=\int_{0}^{x}U_\infty V_\infty V_\infty^\prime ds=0\quad \text{on }[0,x_1],
	\end{equation}
	where we have used \eqref{Uinfty Vinfty x}. Thus, $V_\infty^\prime(x)\equiv0$ on $[0,x_1]$. Utilizing \eqref{Vinfty xxUinfty Vinfty} again, one obtains $U_\infty V_\infty(x)\equiv0$ on $[0,x_1]$, which along with $U_\infty(x_2)>0$ and the monotonicity of $U_\infty(x)$, gives rise to $V_\infty(x)=0$ on $[x_2,x_1]$. Since $V_\infty(x)$ is monotone increasing, we also have $V_\infty(x)\equiv0$ on $[0,x_1]$. Now applying the standard energy estimate for the equation \eqref{Vinfty xxUinfty Vinfty}, we get
	\begin{equation}\nonumber
		V_{\infty}(x)\equiv0 \text{ on }[0,L],
	\end{equation}
	which contradicts to $V_{\infty}(L)=b$. Therefore, we obtain that $U_\infty(x)\equiv0$ on $[0,x_1]$. Now noting $x_1>L-\frac{m}{K}$, we have
	\begin{equation}\nonumber
		\int_{0}^{L} U_{\infty}(x)\mathrm{d}x=\int_{x_1}^{L} U_{\infty}(x)\mathrm{d}x\leq K\left( L-x_1\right)<m.
	\end{equation}
	This contradicts to the constraint $ \int_{0}^{L} U_{\infty}(x)\mathrm{d}x=m$ and we thus show that $$U_{\infty}(x)\equiv K \text{ on } \left(L-\frac{m}{K},L\right].$$

	If there exists	an $x_3\in \left[0,L-\frac{m}{K}\right)$ such that $U_\infty(x_3)>0$, then by  the monotonicity of $U_\infty(x)$, we have $U_\infty(x)>0$ on $[x_3,L]$. Now,
	\begin{equation}
		\begin{aligned}
			\int_{0}^{L} U_{\infty}(x)\mathrm{d}x&>\int_{x_3}^{L} U_{\infty}(x)\mathrm{d}x=\int_{x_3}^{L-\frac{m}{K}} U_{\infty}(x)\mathrm{d}x+\int_{L-\frac{m}{K}}^{L} U_{\infty}(x)\mathrm{d}x\\&=\int_{x_3}^{L-\frac{m}{K}} U_{\infty}(x)\mathrm{d}x+K\cdot\frac{m}{K}\\&>m,		\end{aligned}
	\end{equation}
	which contradicts to $ \int_{0}^{L} U_{\infty}(x)\mathrm{d}x=m$. Therefore, we prove that $$U_{\infty}(x)\equiv 0 \text{ on } \left[0,L-\frac{m}{K}\right).$$

	It remains to show the profile of $V$. In view of \eqref{Vinfty xxUinfty Vinfty}-\eqref{Uinfty}, we find that $V_{\infty}(x)\in W^{2,\infty}(0,L)$ satisfies
	\begin{equation}\label{Vinfty}
		\begin{cases}
			V_\infty^{\prime\prime}=KV_{\infty},\quad &x\in \left(L-\frac{m}{K},L \right),\\
		V_\infty^{\prime\prime}=0 ,\quad &x\in \left(0,L-\frac{m}{K} \right),	\\
			V'_{\infty}(0)=0, \ V_{\infty}(L)=b.
		\end{cases}
	\end{equation}	
	When $x\in \left(0,L-\frac{m}{K} \right)$, due to 	$V'_{\infty}(0)=0$, we have
	\begin{equation}\nonumber
		V_{\infty}(x)\equiv \text{const.}\quad x\in \left[0,L-\frac{m}{K} \right), \text{ and } V'_{\infty}(L-\frac{m}{K})=0.	
	\end{equation} On the other hand, we calculate $V_{\infty}(x)$ on $\left(L-\frac{m}{K},L \right)$ from the first equation of \eqref{Vinfty} and boundary conditions $V_{\infty}(L)=b$ and $V'_{\infty}(L-\frac{m}{K})=0$ that
	\begin{equation}\label{eq35}
		V_{\infty}(x)=C_1 \mathrm{e}^{-\sqrt{K} x}+C_2 \mathrm{e}^{\sqrt{K} x},\quad x \in\left[L-\frac{m}{K}, L\right],
	\end{equation}
	where $C_1=b\mathrm{e}^{\sqrt{K}L}/(1+\mathrm{e}^{\frac{2m}{\sqrt{K}}})$ and $C_2=b\mathrm{e}^{-\sqrt{K}L+\frac{2m}{\sqrt{K}}}/(1+\mathrm{e}^{\frac{2m}{\sqrt{K}}})$. Since $V_\infty(x)$ is continuous, we get from \eqref{eq35} that
	\begin{equation}\label{eq36}
		V_{\infty}(x)\equiv V_{\infty}(L-\frac{m}{K})=\frac{2b\mathrm{e}^{\frac{m}{\sqrt{K}}}}{1+\mathrm{e}^{\frac{2m}{\sqrt{K}}}} ,\quad x\in \left[0,L-\frac{m}{K} \right].	
	\end{equation}
	Combining \eqref{eq35} and \eqref{eq36}, we get the profile of $V$ as $\chi \rightarrow\infty$.
\end{proof}

\section{Nonlinear stability of the steady state}\label{nonlear stability}

In this section, we shall study the nonlinear stability of the steady state $(U,V)$. To this end, we first reformulate the problem by taking the anti-derivative for the density $u$.

\subsection{Reformulation of the problem}

For convenience, we write the first equation of \eqref{original problem} as
\begin{equation}\nonumber
	u_t=\left[\left(q(u)-q^{\prime}(u) u\right) u_x-\chi q(u) u v_x\right]_x\triangleq\left[D(u) u_x-\chi S(u) v_x\right]_x,	
\end{equation}
where
\begin{equation}\label{P H}
	D(u)\triangleq q(u)-q^{\prime}(u) u,\quad S(u)\triangleq q(u) u.
\end{equation}
It is easy to see that $D(u)>0$ and $S(u)>0$ if $u\in(0,K)$, and both are smooth on $[0,K]$. Then, the model \eqref{original problem} can be expressed as
\begin{equation}\label{transformed problem}
	\begin{cases}u_t=\left[D(u) u_x-\chi S(u) v_x\right]_x, & x \in(0,L), \\ v_t=v_{x x}-u v, & x \in(0,L),\\ \left( u,v\right)\left(x_0,0\right)=\left(u_0,v_0\right)(x),& x \in(0,L),\end{cases}
\end{equation}
and the steady state satisfies	
\begin{equation}\label{transformed steady state problem plus}
	\begin{cases}{\left[D(U) U^\prime-\chi S(U)V^\prime\right]^\prime=0,} & x \in(0,L), \\ V^{\prime\prime}=U V, & x \in(0,L).\end{cases}
\end{equation}
Recalling the no-flux boundary condition for $u$, we know that the cell mass is conserved, which along with the fact $\int_0^L U(x) d x=\int_0^L u_0(x) d x$ implies that
\begin{equation}\label{zero mass}
	\int_0^L \left(u(x,t)-U(x)\right) d x=0, \ \ \forall t>0.
\end{equation}
We thus employ the anti-derivative technique to study the asymptotic stability of steady state $(U,V)$. Define
$$
\phi(x, t)\triangleq\int_0^x(u(y, t)-U(y)) \mathrm{d} y, \quad \psi(x, t)\triangleq v(x, t)-V(x),
$$
that is
\begin{equation}\label{decompose}
	u(x, t)=\phi_x(x, t)+U(x),\quad v(x, t)=\psi(x, t)+V(x).
\end{equation}	
Now substituting \eqref{decompose} into \eqref{transformed problem}, integrating the equation over $(0,x)$, and using \eqref{transformed steady state problem plus}, we derive the equation of $(\phi,\psi)$:	
\begin{equation}\label{phipsi equ}
	\begin{cases}
		\phi_t=&D(U) \phi_{x x}+\left(D'(U)U^\prime-\chi S'(U)V^\prime\right)\phi_x
		-\chi S(U)\psi_{ x}\\&+\left( D\left( \phi_x+U\right)-D(U)\right)\phi_{xx}+\left(D\left( \phi_x+U\right)-D(U)-D'(U)\phi_{ x}\right)U^\prime\\&-\chi\left(S\left( \phi_x+U\right)-S(U) \right)\psi_{ x}-\chi\left(S\left( \phi_x+U\right)-S(U)-S'(U)\phi_{ x}\right)V^\prime,\\
		\psi_t=&\psi_{x x}-U \psi-V \phi_x-\phi_x \psi,
	\end{cases}
\end{equation}where we have used the decomposition
\begin{equation}\nonumber
	\begin{split}
		&D\left( \phi_x+U\right)-D(U)=D'(U)\phi_{ x}+\left(D\left( \phi_x+U\right)-D(U)-D'(U)\phi_{ x}	 \right),\\
		&S\left( \phi_x+U\right)-S(U)=S'(U)\phi_{ x}+\left(S\left( \phi_x+U\right)-S(U)-S'(U)\phi_{ x}\right).
\end{split}\end{equation}
The initial value of \eqref{phipsi equ} is
\begin{equation}\label{phi0psi0}
	(\phi, \psi)(x, 0)\triangleq\left(\phi_0, \psi_0\right)(x)=\left(\int_0^x\left(u_0(y)-U(y)\right) \mathrm{d} y, v_0-V\right).
\end{equation}	
Owing to \eqref{original boundary condition}, \eqref{steady state boundary conditon} and \eqref{zero mass}, the boundary conditions	of  \eqref{phipsi equ} are given by
\begin{equation}\label{phipsi boundary condition}
	\left(\phi, \psi_x\right)(0, t)=(0,0),\quad(\phi, \psi)(L, t)=(0,0).
\end{equation}

\subsection{\textit{A priori} estimates}	
We next establish the global existence of solutions for the problem \eqref{phipsi equ}-\eqref{phipsi boundary condition}. For any $T>0$, we search for solutions of \eqref{phipsi equ}-\eqref{phipsi boundary condition} in the following space
\begin{equation}\nonumber
	\begin{aligned}
		X(0, T)\triangleq\{(\phi, \psi) \mid \phi & \in C\left([0, T] ; H_0^1 \cap H^2\right) \cap L^2\left(0, T ; H^3\right), \\
		\psi & \left.\in C\left([0, T] ; H^1\right) \cap \in L^2\left(0, T ; H^2\right)\right\}.
	\end{aligned}
\end{equation}
We have the following global well-posedness result.

\begin{proposition}\label{global existence} Let $\chi>0$ be fixed. Assume that \textbf{(A)} holds.
	Suppose $\phi_0\in H_0^1 \cap H^2$ and $\psi_0\in H^1$ with $\psi_0(L)=0$. Then there exists a positive constant $\epsilon$ such that if $\left\|\phi_0\right\|_{H^2}^2+\left\|\psi_0\right\|_{H^1}^2\leq \epsilon$, then the problem \eqref{phipsi equ}-\eqref{phipsi boundary condition} has a unique global solution $(\phi,\psi)\in X(0,\infty)$ satisfying for all $t\geq 0$,
	\begin{equation}\label{global estimate}
		\|\phi(\cdot, t)\|_{H^2}^2+\|\psi(\cdot, t)\|_{H^1}^2 \leq C\mathrm{e}^{-\alpha t}\left(\|\phi_0\|_{H^2}^2+\|\psi_0\|_{H^1}^2 \right),
	\end{equation}
	for some constants $\alpha>0$ and $C>0$ independent of $t$.
\end{proposition}

The global existence of $(\phi,\psi)$, as stated in Proposition \ref{global existence}, can be treated by the energy method based on local existence with the \textit{a priori} estimates.

\begin{proposition}[Local existence]\label{pro local existence}
	Let $\chi>0$ be fixed. Assume that \textbf{(A)} holds. Suppose $\phi_0\in H_0^1 \cap H^2$ and $\psi_0\in H^1$ with $\psi_0(L)=0$. For any $\delta_0>0$, there exists a positive constant $T_0>0$ depending on $\delta_0$ such that if $\|\phi_0\|_{H^2}^2+\|\psi_0\|_{H^1}^2\leq \delta_0$, then the problem \eqref{phipsi equ}-\eqref{phipsi boundary condition} has a unique global solution $(\phi,\psi)\in X(0,T_0)$ satisfying
	\begin{equation}\label{local estimate}
		\|\phi(\cdot, t)\|_{H^2}^2+\|\psi(\cdot, t)\|_{H^1}^2 \leq 2\left(\|\phi_0\|_{H^2}^2+\|\psi_0\|_{H^1}^2 \right),
	\end{equation}
	for any $0\leq t\leq T_0$.
\end{proposition}

The local existence of solutions to the problem \eqref{phipsi equ}-\eqref{phipsi boundary condition} can be shown by the standard fixed point theorem. To prove Proposition \ref{global existence}, we only need to establish the \textit{a priori} estimates.

\begin{proposition}[\emph{A priori} estimates] \label{proposition a priori estimates}  Let $\chi>0$ be fixed. Assume that \textbf{(A)} holds.
	For any $T>0$ and any solution $\left(\phi,\psi\right)\in X(0,T)$ of the problem \eqref{phipsi equ}-\eqref{phipsi boundary condition}, there exists a constant $\delta_1>0$, independent of $T$, such that if
	\begin{equation}\label{N(t)}
		\left\|\phi\left(\cdot,t\right)\right\|_{H^2}^2+\left\|\psi\left(\cdot,t\right)\right\|_{H^1}^2\leq \delta_1,\ \text{ for all } t\in[0,T],	
	\end{equation}
	then we have
	\begin{equation}\label{a priori estimate1}
		\|\phi(\cdot, t)\|_{H^2}^2+\|\psi(\cdot, t)\|_{H^1}^2 \leq C\mathrm{e}^{-\alpha t}\left(\|\phi_0\|_{H^2}^2+\|\psi_0\|_{H^1}^2 \right)\ \ \text{ for all } t\in[0,T],	
	\end{equation}and
	\begin{equation}\label{a priori estimate3}
		\int_0^t\mathrm{e}^{\alpha \tau}\left(\|\phi\|_{H^3}^2+\|\psi\|_{H^2}^2\right) \mathrm{d} \tau \leq C\left(\left\|\phi_0\right\|_{H^2}^2+\left\|\psi_0\right\|_{H^1}^2\right)\ \ \text{ for all } t\in[0,T],
	\end{equation}
	where $\alpha$ and $C$ are positive constants independent of $T$.
\end{proposition}

Before establishing the \textit{a priori} estimates in Proposition \ref{proposition a priori estimates},
we note by \eqref{N(t)} and Sobolev inequality, it holds
\begin{equation}
	\sup_{\tau\in [0,t]}\left\{\left\|\phi(\cdot,\tau)\right\|_{L^{\infty}},
	\left\|\phi_x(\cdot,\tau)\right\|_{L^{\infty}}, \left\|\psi(\cdot,\tau)\right\|_{L^{\infty}}\right\} \leq C\sqrt{\delta_1} \text{	 for any }  t\in[0,T],
\end{equation}
where $C>0$ is a constant independent of $\delta_1$ and $T$.

Now let us turn to the \emph{a priori} estimates for $(\phi,\psi)$. We begin with the following $L^2$ estimate.

\begin{lemma}\label{lemma L2} Assume that the conditions of Proposition \ref{proposition a priori estimates} hold.
	For any solution $(\phi,\psi)\in X(0,T)$ of the problem \eqref{phipsi equ}-\eqref{phipsi boundary condition} satisfying \eqref{N(t)}, it holds that
	\begin{equation}\label{L2 estimate}
		\mathrm{e}^{\alpha t}(\left\|\phi\left(\cdot,t\right)\right\|_{L^2}^2+\left\|\psi\left(\cdot,t\right)\right\|_{L^2}^2)+\int_0^t\mathrm{e}^{\alpha \tau}\left(\left\|\phi_x\right\|_{L^2}^2+\left\|\psi_x\right\|_{L^2}^2\right) \mathrm{d} \tau\leq C\left(\left\|\phi_0\right\|_{L^2}^2+\left\|\psi_0\right\|_{L^2}^2\right),
	\end{equation}
	for any $t\in[0,T]$, provided that $\sqrt{\delta_1}$ is suitably small and $\delta_{1}\leq\delta_0$, where $\alpha$ and $C$ are positive constants independent of $t$.
\end{lemma}

\begin{proof} Because $0<U(x)<K$ for all $x\in[0,L]$, we know that $S(U)$ has positive lower bound. Then multiplying the first equation of \eqref{phipsi equ} by $\frac{\phi}{S(U)}$ and the second one by $\chi \frac{\psi}{V}$, adding them and integrating by parts, noting that	
	\begin{equation}\nonumber
		\int \frac{D(U)}{S(U)}\phi\phi_{xx}=-\int\frac{D(U)}{S(U)}\phi_x^2-\int\left(\frac{D(U)}{S(U)} \right)_x\phi\phi_x		
	\end{equation}
	and
	\begin{equation}\nonumber
		\chi\int\frac{\psi\psi_{xx}}{V}=-\chi\int\frac{\psi_x^2}{V}+\chi\int\frac{V^\prime}{V^2}\left(\frac{\psi^2}{2} \right)_x=-\chi\int\frac{\psi_x^2}{V}-		 \chi\int\left(\frac{V^\prime}{2V^2}\right)_x\psi^2,
	\end{equation}
	we obtain
	\begin{equation}\label{eq2}
		\begin{aligned}
			&\frac{1}{2}\frac{\mathrm{d}}{\mathrm{d}t}\int\left(\frac{\phi^2}{S(U)}+\chi\frac{\psi^2}{V} \right)+\int\frac{D(U)}{S(U)}\phi_x^2+\chi \int\frac{\psi_x^2}{V}+\chi\int\left[\left(\frac{V^\prime}{2V^2}\right)_x+\frac{U}{V}\right]\psi^2\\&=\int\left[-\left(\frac{D(U)}{S(U)} \right)_x+\frac{D'(U)U^\prime-\chi S'(U)V^\prime }{S(U)}\right]\phi\phi_{ x}+\int \frac{\phi\phi_{x x}}{S(U)}\left( D\left( \phi_x+U\right)-D(U)\right)\\&\quad+\int \frac{U^\prime}{S(U)}\left(D\left( \phi_x+U\right)-D(U)-D'(U)\phi_{ x}\right)\phi-\chi\int \frac{\phi\psi_{x}}{S(U)}\left(S\left( \phi_x+U\right)-S(U) \right)\\&\quad-\chi\int \frac{V^\prime}{S(U)}\left(S\left( \phi_x+U\right)-S(U)-S'(U)\phi_{ x}\right)\phi-\chi\int\frac{\phi_x\psi^2}{V}.
		\end{aligned}
	\end{equation}
	By virtue of \eqref{inequality} and the second equation of \eqref{transformed steady state problem plus}, it holds that
	\begin{equation}\label{eq3}
		\left(\frac{V^\prime}{2V^2}\right)_x+\frac{U}{V} =\frac{V^{\prime\prime  }}{2V^2}-\frac{ \left(V^\prime \right)^2}{V^3}+\frac{U}{V} =\frac{3}{2} \frac{U}{V}- \frac{\left(V^\prime \right)^2}{V^3}\geq \frac{U}{2V}.
	\end{equation}
	In view of the first equation of \eqref{transformed steady state problem plus}, we have
	\begin{equation}\label{eq4}
		-\left(\frac{D(U)}{S(U)} \right)_x+\frac{D'(U)U^\prime-\chi S'(U)V^\prime }{S(U)}=0.
	\end{equation}
	Owing to the fact that  $\left\|\phi_x(\cdot,t)\right\|_{L^{\infty}} \leq C\sqrt{\delta_1}$, by Taylor's expansion, when $\sqrt{\delta_1}$ is small such that $C\sqrt{\delta_1}<\frac{1}{2}(K-\|U\|_{C[0,L]})$, we have
	\begin{equation}\label{new-4.19}
		|D\left( \phi_x+U\right)-D(U)|\leq C|\phi_x|, \ |S\left( \phi_x+U\right)-S(U)|\leq C|\phi_x|,
	\end{equation}
	and
	\begin{equation}\label{new-4.20}
		|D\left( \phi_x+U\right)-D(U)-D'(U)\phi_{ x}|\leq C|\phi_x|^2, \ |S\left( \phi_x+U\right)-S(U)-S'(U)\phi_{ x}|\leq C|\phi_x|^2.
	\end{equation}
	Recalling $S(U)$ has a positive lower bound, and noting $\left\|\phi_{xx}\right\|_{L^{2}} \leq C\sqrt{\delta_1}$, it follows from Sobolev inequality $\|f\|_{L^{\infty}} \leq C\|f_x\|_{L^2}$ for any $f\in H_0^1$ and H\"{o}lder's inequality that
	\begin{equation}\label{4.19}
		\left|\int \frac{\phi\phi_{x x}}{S(U)}\left( D\left( \phi_x+U\right)-D(U)\right)\right|\leq C\left\|\phi\right\|_{L^{\infty}}\left\|\phi_x\right\|_{L^{2}}\left\|\phi_{xx}\right\|_{L^{2}}\leq C \sqrt{\delta_1} \left\|\phi_{x}\right\|_{L^{2}}^2.
	\end{equation}
	Since $\left\|\phi\right\|_{L^{\infty}} \leq C\|\phi_x(\cdot,t)\|_{L^2} \leq C\sqrt{\delta_1}$, by  H\"{o}lder's inequality, we deduce that
	\begin{equation}\label{4.20}
		\left|\int \frac{U^\prime}{S(U)}\left(D\left( \phi_x+U\right)-D(U)-D'(U)\phi_{ x}\right)\phi\right|\leq C \int |\phi|\phi_{ x}^2\leq C\sqrt{\delta_1}\int\phi_x^2,
	\end{equation}
	\begin{equation}\label{4.21}
		\chi\left|\int \frac{\phi\psi_{x}}{S(U)}\left(S\left( \phi_x+U\right)-S(U) \right)\right|\leq C \int |\phi\phi_{ x}\psi_x|\leq C\sqrt{\delta_1}\left(\int\phi_x^2+\int\psi_{ x}^2 \right),		
	\end{equation}
	and that
	\begin{equation}\label{eq7}
		\chi\left|\int \frac{V^\prime}{S(U)}\left(S\left( \phi_x+U\right)-S(U)-S'(U)\phi_{ x}\right)\phi\right|\leq C \int |\phi|\phi_{ x}^2\leq C\sqrt{\delta_1}\int\phi_x^2.
	\end{equation}	
	Similarly, thanks to the lower boundedness of $V$, it holds
	\begin{equation}\label{eq6}
		\chi\left|\int\frac{\phi_x\psi^2}{V}\right|\leq C\sqrt{\delta_1} \int \psi^2.	 	
	\end{equation}
	Substituting \eqref{eq3}-\eqref{eq4}, \eqref{4.19}-\eqref{eq6} into \eqref{eq2}, after choosing $\sqrt{\delta_1}$ suitably small and $\delta_1\leq \delta_0$, we obtain
	\begin{equation}\label{eq8}
		\frac{\mathrm{d}}{\mathrm{d}t}\int\left(\frac{\phi^2}{S(U)}+\chi\frac{\psi^2}{V} \right)+\int\frac{D(U)}{S(U)}\phi_x^2+\chi \int\frac{\psi_x^2}{V}+\chi\int \frac{U}{V}\psi^2\leq 0.
	\end{equation}
	Multiplying \eqref{eq8} by $\mathrm{e}^{\alpha t}$ and integrating the inequality over $(0,t)$, we have
	\begin{equation}\label{4.25}
		\begin{split}
			&\mathrm{e}^{\alpha t}\int\left(\frac{\phi^2}{S(U)}+\chi\frac{\psi^2}{V} \right)+\int_0^t\mathrm{e}^{\alpha \tau}\int\left(\frac{D(U)}{S(U)}\phi_x^2+\chi\frac{\psi_x^2}{V}+\chi\frac{U}{V}\psi^2-\alpha\frac{\phi^2}{S(U)}-\alpha\chi\frac{\psi^2}{V}\right)\\
			&\leq\int\left(\frac{\phi_0^2}{S(U)}+\chi\frac{\psi_0^2}{V} \right).
	\end{split}\end{equation}
	Therefore, noting $\frac{D(U)}{S(U)}$ has a positive lower bound, when $\alpha\ll1$, the desired estimate \eqref{L2 estimate} follows from Poincar\'{e}'s inequality.	
\end{proof}

We next derive the estimate for $(\phi_{x},\psi_{ x})$.
\begin{lemma}\label{lemma H1}
	Under the conditions of Lemma \ref{lemma L2}, for any $t\in[0,T]$, it holds that
	\begin{equation}\label{H^1 estimate}
		\begin{aligned}
			&\mathrm{e}^{\alpha t}\left( \left\|\phi_x(\cdot,t)\right\|_{L^2}^2+\left\|\psi_x(\cdot,t)\right\|_{L^2}^2\right)+\int_0^t\mathrm{e}^{\alpha\tau}\left(\left\|\phi_\tau\right\|_{L^2}^2+\left\|\psi_\tau\right\|_{L^2}^2\right) \mathrm{d} \tau\\&\quad\leq C\left(\left\|\phi_{0}\right\|_{H^1}^2+\left\|\psi_{0}\right\|_{H^1}^2 \right) +C\sqrt{\delta_1}\int_0^t \mathrm{e}^{\alpha\tau} \left\|\phi_{x\tau}\right\|_{L^2}^2\mathrm{d} \tau,
		\end{aligned}	
	\end{equation}
	provided $\sqrt{\delta_1}$ is suitably small, where $\alpha$ is the constant in \eqref{L2 estimate}.
\end{lemma}

\begin{proof}
	Because $D(U)$ has a positive lower bound, we multiply the first equation of \eqref{phipsi equ} by $\frac{ \phi_{t}}{D(U)}$ and the second one by $\psi_{t}$, and integrate the equations on $(0,L)$ to obtain
	\begin{equation}\label{eq9}
		\begin{split}
			&\frac{1}{2}\frac{\mathrm{d}}{\mathrm{d}t}\int \left(\phi_{x}^2+\psi_{x}^2+U \psi^2 \right)+\int\left(\frac{\phi_{t}^2}{D(U)}+\psi_{t}^2 \right)\\&=\int \frac{\phi_{ x}\phi_t}{D(U)}\left( D'(U)U^\prime-\chi S'(U)V^\prime\right)-\chi \int \frac{S(U)}{D(U)}\psi_x\phi_t+\int \frac{\phi_{x x}\phi_t}{D(U)}\left( D\left( \phi_x+U\right)-D(U)\right)\\&\quad-\chi\int \frac{\psi_{x}\phi_t}{D(U)}\left(S\left( \phi_x+U\right)-S(U) \right)+\int \frac{U^\prime}{D(U)}\left(D\left( \phi_x+U\right)-D(U)-D'(U)\phi_{ x}\right)\phi_t\\&\quad-\chi\int \frac{V^\prime}{D(U)}\left(S\left( \phi_x+U\right)-S(U)-S'(U)\phi_{ x}\right)\phi_t-\int V\phi_{x}\psi_t-\int\phi_x\psi\psi_t.
		\end{split}	
	\end{equation}	
	By Cauchy-Schwarz inequality, it holds that
	\begin{equation}\label{eq10}
		\begin{aligned}
			\left|\int \frac{\phi_{ x}\phi_t}{D(U)}\left( D'(U)U^\prime-\chi S'(U)V^\prime\right)\right|+\chi \left|\int \frac{S(U)}{D(U)}\psi_x\phi_t\right|\leq \frac{1}{4}\int\frac{\phi_{t}^2}{D(U)}+C\int\left(\phi_{x}^2+\psi_{x}^2 \right), 	
		\end{aligned}
	\end{equation}
	and that
	\begin{equation}
		\left|\int V\phi_{x}\psi_t\right|\leq \frac{1}{4}\int\psi_{t}^2+C\int\phi_{x}^2.
	\end{equation}	
	In view of the fact that $\left\|\phi_{xx}(\cdot,t)\right\|_{L^{2}} \leq \sqrt{\delta_1}$, by H\"{o}lder's inequality, the third term on the right hand side of \eqref{eq9} can be estimated as	
	\begin{equation}
		\begin{aligned}
			\left|\int \frac{\phi_{x x}\phi_t}{D(U)}\left( D\left( \phi_x+U\right)-D(U)\right)\right|&\leq C\|\phi_t\|_{L^{\infty}}\|\phi_{x}\|_{L^{2}}\|\phi_{xx}\|_{L^{2}}\\&\leq C\sqrt{\delta_1}	\|\phi_{xt}\|_{L^{2}}\|\phi_{x}\|_{L^{2}}\\&\leq 	C\sqrt{\delta_1}\int \phi_{xt}^2 +C\sqrt{\delta_1}\int \phi_{x}^2,	
		\end{aligned}
	\end{equation}where we have used \eqref{new-4.19} in the first inequality and Sobolev inequality in the second inequality. Similarly, using  $\left\|\phi_x(\cdot,t)\right\|_{L^{\infty}} \leq C\sqrt{\delta_1}$, the fourth term satisfies
	\begin{equation}
		\chi\left|\int \frac{\psi_{x}\phi_t}{D(U)}\left(S\left( \phi_x+U\right)-S(U) \right)\right|\leq C\sqrt{\delta_1}\int\frac{\phi_{t}^2}{D(U)}+C\sqrt{\delta_1}\int\psi_{x}^2.
	\end{equation}
	Utilizing \eqref{new-4.20}, we deduce that
	\begin{equation}
		\left|\int \frac{U^\prime}{D(U)}\left(D\left( \phi_x+U\right)-D(U)-D'(U)\phi_{ x}\right)\phi_t\right|\leq C\sqrt{\delta_1}\int\frac{\phi_{t}^2}{D(U)}+C\sqrt{\delta_1}\int\phi_{x}^2,
	\end{equation}	
	and
	\begin{equation}
		\chi\left|\int \frac{V^\prime}{D(U)}\left(S\left( \phi_x+U\right)-S(U)-S'(U)\phi_{ x}\right)\phi_t\right|\leq C\sqrt{\delta_1}\int\frac{\phi_{t}^2}{D(U)}+C\sqrt{\delta_1}\int\phi_{x}^2.
	\end{equation}
	Thanks to $\left\|\psi(\cdot,t)\right\|_{L^{\infty}} \leq C\sqrt{\delta_1}$ and H\"{o}lder's inequality, the last term on the right hand side of \eqref{eq9} can be estimated as
	\begin{equation}\label{eq11}
		\left|\int\phi_x\psi\psi_t\right|\leq C \sqrt{\delta_1}\int|\phi_{x}\psi_t|\leq \sqrt{\delta_1}\int \psi_t^2+C\sqrt{\delta_1}\int\phi_x^2.
	\end{equation}
	Now substituting \eqref{eq10}-\eqref{eq11} into \eqref{eq9} and choosing $\sqrt{\delta_1}$ small enough, we get
	\begin{equation}\label{eq12}
		\frac{\mathrm{d}}{\mathrm{d}t}\int \left(\phi_{x}^2+\psi_{x}^2+U \psi^2 \right)+\int\left(\phi_{t}^2+\psi_{t}^2 \right)\leq C\int\left(\phi_{x}^2+\psi_x^2\right)+C\sqrt{\delta_1}\int \phi_{xt}^2.
	\end{equation}
	Multiplying \eqref{eq12} by $\mathrm{e}^{\alpha t}$, integrating the inequality in $t$ and using \eqref{L2 estimate}, one obtains \eqref{H^1 estimate}.
\end{proof}

In what follows, we derive the higher order estimates for $(\phi,\psi)$.

\begin{lemma}\label{lemma H2}
	Under the same conditions of Lemma \ref{lemma L2}, it holds for any $t\in[0,T]$ that
	\begin{equation}\label{H2 eistimate}
		\begin{split}\mathrm{e}^{\alpha t}(\|\phi(\cdot, t)\|_{H^2}^2+\|\psi(\cdot, t)\|_{H^1}^2 )+\int_0^t\mathrm{e}^{\alpha \tau}\left(\|\phi\|_{H^3}^2+\|\psi\|_{H^2}^2\right) \mathrm{d} \tau\leq C\left(\|\phi_0\|_{H^2}^2+\|\psi_0\|_{H^1}^2 \right),
	\end{split}\end{equation}
	provided $\sqrt{\delta_1}$ is suitably small.
\end{lemma}

\begin{proof}
	Differentiating \eqref{phipsi equ} with respect to $t$, we get
	\begin{equation}\label{phittpsitt}
		\begin{cases}
			\phi_{tt}=&D(U)\phi_{xxt}
			+\left(D'\left( \phi_x+U\right)U^\prime-\chi S'(U)V^\prime\right)\phi_{xt}+\left( D\left( \phi_x+U\right)-D(U)\right)\phi_{xxt}\\&+D'\left( \phi_x+U\right)\phi_{x t}\phi_{xx}-\chi S\left( \phi_x+U\right)\psi_{ xt}-\chi S'\left( \phi_x+U\right)\phi_{ xt}\psi_x\\&-\chi\left(S'\left( \phi_x+U\right)-S'(U)\right)V^\prime\phi_{ xt},\\
			\psi_{t t}=&\psi_{x x t}-U \psi_t-V \phi_{x t}-\phi_{x t} \psi-\phi_x \psi_t.
		\end{cases}
	\end{equation}
	Multiplying the first equation of \eqref{phittpsitt} by $\phi_t$ and integrating the equation on $(0,L)$, we get
	\begin{equation}\label{eq13}
		\begin{aligned}
			\frac{1}{2}\frac{\mathrm{d}}{\mathrm{d}t}\int \phi_{t}^2+\int D(U)\phi_{xt}^2&=-\chi\int S\left( \phi_x+U\right) \psi_{ xt}\phi_{t}-\chi\int S'\left( \phi_x+U\right)\phi_{ xt}\psi_x\phi_{t}\\&\quad-\chi\int S'(U)V^\prime\phi_{xt}\phi_{t}-\int\left( D\left( \phi_x+U\right)-D(U)\right)\phi_{xt}^2\\&\quad-\chi\int\left(S'\left( \phi_x+U\right)-S'(U)\right)V^\prime\phi_{ xt}\phi_{t}.
		\end{aligned}	
	\end{equation}
	We now estimate the terms on the right hand side of \eqref{eq13}. After integrating by parts, we have
	\begin{equation*}
		\begin{aligned}
			&-\chi\int S\left( \phi_x+U\right)\psi_{ xt}\phi_{t}\\&=\chi\int S\left( \phi_x+U\right)\psi_{t}\phi_{xt}+\chi\int S'\left( \phi_x+U\right)U^\prime\phi_{t}\psi_{t}+\chi\int S'\left( \phi_x+U\right)\phi_{xx}\phi_{t}\psi_{t}.
		\end{aligned}
	\end{equation*} Because $D(U)$ has a positive lower bound, by Young's inequality, we have
	\[\chi\left|\int S\left( \phi_x+U\right)\psi_{t}\phi_{xt}\right|\leq\frac{1}{4}\int D(U)\phi_{xt}^2+C\int\psi_t^2,\]
	and
	\[\chi\left|\int S'\left( \phi_x+U\right)U^\prime\phi_{t}\psi_{t}\right|\leq C\int(\phi_t^2+\psi_t^2).\]
	By virtue of the fact that $\left\|\phi_{xx}\right\|_{L^{2}} \leq C\sqrt{\delta_1}$, it holds that
	\begin{equation}\nonumber
		\begin{aligned}
			\chi\left|\int S'\left( \phi_x+U\right)\phi_{xx}\phi_{t}\psi_{t}\right|	&\leq C	\left\|\phi_{t}\right\|_{L^{\infty}}\left\|\phi_{xx}\right\|_{L^{2}}\left\|\psi_{t}\right\|_{L^{2}}\\&\leq C \sqrt{\delta_1}\left\|\phi_{xt}\right\|_{L^{2}}\left\|\psi_{t}\right\|_{L^{2}}\\&\leq\sqrt{\delta_1}\int \phi_{xt}^2+C\sqrt{\delta_1}\int \psi_{t}^2.
		\end{aligned}
	\end{equation}
	Thus,
	\begin{equation}\label{eq14}
		\begin{split}\chi\left|\int S\left( \phi_x+U\right) \psi_{ xt}\phi_{t}\right|\leq\frac{1}{4}\int D(U)\phi_{xt}^2+\sqrt{\delta_1}\int\phi_{xt}^2+C\int \left(\phi_{t}^2+\psi_{t}^2 \right).\end{split}	
	\end{equation}
	Utilizing $\left\|\psi_{x}(\cdot,t)\right\|_{L^{2}} \leq C\sqrt{\delta_1}$ and Sobolev inequality, the second term on the right hand side of \eqref{eq13} gives rise to
	\begin{equation}\label{eq16}
		\begin{aligned}
			\chi\left|\int S'\left(\phi_x+U\right)\phi_{xt}\psi_{x}\phi_{t}\right|\leq C \left\|\phi_{t}\right\|_{L^{\infty}}\left\|\psi_{x}\right\|_{L^{2}}\left\|\phi_{xt}\right\|_{L^{2}}\leq C \sqrt{\delta_1}\int \phi_{xt}^2.\end{aligned}
	\end{equation}
	By \eqref{V_xU_xV_xx,U_xx} and Cauchy-Schwarz inequality, we have
	\begin{equation}\label{eq15}
		\chi \left|\int S'(U)V^\prime\phi_{xt}\phi_{t}\right|\leq \frac{1}{4} \int D(U)\phi_{xt}^2+C\int \phi_{t}^2.		
	\end{equation}
	Thanks to $\left\|\phi_{x}\right\|_{L^{\infty}} \leq C\sqrt{\delta_1}$, by Taylor's expansion and the Cauchy-Schwarz inequality, we get
	\begin{equation}\label{4.43}
		\begin{split}&\left|\int\left( D\left( \phi_x+U\right)-D(U)\right)\phi_{xt}^2\right|+\chi\left|\int\left(S'\left( \phi_x+U\right)-S'(U)\right)V^\prime\phi_{ xt}\phi_{t}\right|\\&\leq \sqrt{\delta_1}\int \phi_{xt}^2+C\sqrt{\delta_1}\int \phi_{t}^2.	\end{split}	
	\end{equation}
	Inserting \eqref{eq14}-\eqref{4.43} into \eqref{eq13} and choosing $\sqrt{\delta_1}$ small enough, we get
	\begin{equation}\label{eq17}
		\frac{\mathrm{d}}{\mathrm{d}t}\int \phi_{t}^2+\int \phi_{xt}^2	\leq C\int \left( \phi_{t}^2+\psi_{t}^2\right).
	\end{equation}
	Multiplying \eqref{eq17} by $\mathrm{e}^{\alpha t}$, and integrating in $t$, by \eqref{H^1 estimate}, we then obtain
	\begin{equation}\label{phit}
		\mathrm{e}^{\alpha t}\int \phi_{t}^2+\int_{0}^{t}\mathrm{e}^{\alpha \tau}\int \phi_{x\tau}^2\leq C \left(\left\|\phi_0\right\|_{H^2}^2+\left\|\psi_0\right\|_{H^1}^2\right), 	
	\end{equation}
	where we have used
	\begin{equation}\nonumber
		\begin{aligned}
			\phi_t|_{t=0}&=D(U) \phi_{0x x}+\left(D'(U)U^\prime-\chi S'(U)V^\prime\right)\phi_{0x}-\chi S(U)\psi_{0x}\\&\quad+\left( D\left( \phi_{0x}+U\right)-D(U)\right)\phi_{0xx}+\left(D\left( \phi_{0x}+U\right)-D(U)-D'(U)\phi_{0 x}\right)U^\prime\\&\quad-\chi\left(S\left( \phi_{0x}+U\right)-S(U) \right)\psi_{0 x}-\chi\left(S\left( \phi_{0x}+U\right)-S(U)-S'(U)\phi_{0 x}\right)V^\prime.
		\end{aligned}
	\end{equation}
	Adding \eqref{phit} with \eqref{L2 estimate} and \eqref{H^1 estimate}, we then arrive at
	\begin{equation}\label{together}
		\begin{aligned}
			&\mathrm{e}^{\alpha t}(\|\phi(\cdot, t)\|_{H^1}^2+\|\psi(\cdot, t)\|_{H^1}^2+\|\phi_t(\cdot, t)\|_{L^2}^2)\\&\quad+\int_0^t\mathrm{e}^{\alpha \tau}\left(\left\|\phi_x\right\|_{L^2}^2+\left\|\psi_x\right\|_{L^2}^2+\left\|\phi_\tau\right\|_{H^1}^2+\left\|\psi_\tau\right\|_{L^2}^2\right) \mathrm{d} \tau\leq C\left(\left\|\phi_0\right\|_{H^2}^2+\left\|\psi_0\right\|_{H^1}^2\right).
		\end{aligned}	
	\end{equation}
	On the other hand, since $D(U)$ has a positive lower bound, the first equation of \eqref{phipsi equ} gives
	\[\begin{split}
		\phi_{x x}^2&\leq C(\phi_{t}^2+\phi_{x}^2+\psi_{x}^2)+C(D\left( \phi_x+U\right)-D(U))^2\phi_{xx}^2+C(D\left( \phi_x+U\right)-D(U)-D'(U)\phi_{ x})^2\\&\quad+C(S\left( \phi_x+U\right)-S(U))^2\psi_{x}^2+C(S\left( \phi_x+U\right)-S(U)-S'(U)\phi_{ x})^2.\end{split}\]
	Thanks to $\left\|\phi_{x}\right\|_{L^{\infty}} \leq C\sqrt{\delta_1}$, if $\delta_1\ll1$, it follows from \eqref{new-4.19}-\eqref{new-4.20} that
	\begin{equation}\nonumber
		\begin{aligned}
			\int \phi_{x x}^2\leq C\int \left( \phi_{t}^2+\phi_{x}^2+\psi_{x}^2\right)+C\delta_1 \int \phi_{xx}^2 +C\delta_1 \int \left(\phi_{x}^2+\psi_{x}^2 \right).
		\end{aligned}
	\end{equation}
	Hence
	\begin{equation}\nonumber
		\int \phi_{x x}^2\leq C	\int \left( \phi_{t}^2+\phi_{x}^2+\psi_{x}^2\right).
	\end{equation}
	This together with \eqref{together} yields that
	\begin{equation}\label{phixx}
		\mathrm{e}^{\alpha t}\|\phi_{x x}\|_{L^2}^2\leq C\left(\left\|\phi_0\right\|_{H^2}^2+\left\|\psi_0\right\|_{H^1}^2\right),
	\end{equation}
	and that
	\begin{equation}\label{int_0^t phixx}
		\int_{0}^{t}\mathrm{e}^{\alpha \tau}\int \phi_{x x}^2\leq C	\int_{0}^{t}\mathrm{e}^{\alpha \tau}\int \left( \phi_{\tau}^2+\phi_{x}^2+\psi_{x}^2\right)\leq C\left(\left\|\phi_0\right\|_{H^2}^2+\left\|\psi_0\right\|_{H^1}^2\right).	\end{equation}
	Similarly, by the second equation of \eqref{phipsi equ}, we get
	\begin{equation}\label{int_0^t psixx}
		\int_{0}^{t}\mathrm{e}^{\alpha \tau}\int \psi_{x x}^2\leq C	\int_{0}^{t}\mathrm{e}^{\alpha \tau}\int \left( \psi_{\tau}^2+\psi^2+\phi_{x}^2\right)\leq C\left(\left\|\phi_0\right\|_{H^2}^2+\left\|\psi_0\right\|_{H^1}^2\right).	
	\end{equation}

	It remains to estimate $\int_{0}^{t}\mathrm{e}^{\alpha \tau}\int\phi_{xxx}^2$. Differentiating \eqref{phipsi equ} with respect to $x$, in view of the fact that $\|\phi_{x}(\cdot,t)\|_{L^{\infty}}\leq C \sqrt{\delta_1}$ and the boundedness of $(U,V)$, we obtain
	\begin{equation}\label{eq21}
		\begin{aligned}
			\int_{0}^{t}\mathrm{e}^{\alpha \tau}\int \phi_{xxx}^2&\leq C\int_{0}^{t}\mathrm{e}^{\alpha \tau}\int\left(\phi_{xt}^2+\phi_{xx}^2+\phi_{x}^2+\psi_{xx}^2+\psi_{x}^2 \right)\\&\quad+C\int_{0}^{t}\mathrm{e}^{\alpha \tau}\cdot\|\phi_{x}\|_{L^{\infty}}^2\int\left(\phi_{xxx}^2+\phi_{xx}^2+\phi_{ x}^2+\psi_{xx}^2+\psi_{x}^2\right)\\&\quad+C\int_{0}^{t}\mathrm{e}^{\alpha \tau}\int\psi_{x}^2\phi_{xx}^2+C\int_{0}^{t}\mathrm{e}^{\alpha \tau}\int\phi_{xx}^4\\&\leq C\delta_1 \int_{0}^{t}\mathrm{e}^{\alpha \tau}\int \phi_{xxx}^2+C\int_{0}^{t}\mathrm{e}^{\alpha \tau}\int\left(\phi_{x}^2+\psi_{x}^2 +\phi_{xx}^2+\psi_{xx}^2+\phi_{xt}^2\right)\\&\quad+C\int_{0}^{t}\mathrm{e}^{\alpha \tau}\int\psi_{x}^2\phi_{xx}^2+C\int_{0}^{t}\mathrm{e}^{\alpha \tau}\int\phi_{xx}^4.
		\end{aligned}
	\end{equation}
	By virtue of $\|\phi_{xx}(\cdot,t)\|_{L^{2}}\leq C \sqrt{\delta_1}$, $\|\psi_{x}(\cdot,t)\|_{L^{2}}\leq C \sqrt{\delta_1}$ and the following Sobolev inequality $$\|f\|_{L^{\infty}} \leq C\left(\|f\|_{L^2}^{\frac{1}{2}}\left\|f_x\right\|_{L^2}^{\frac{1}{2}}+\|f\|_{L^2}\right) \ \text{ for any } f\in H^1,$$ it holds that
	\begin{equation}\nonumber
		\begin{aligned}
			\int_{0}^{t}\mathrm{e}^{\alpha \tau}\int\psi_{x}^2\phi_{xx}^2+\int_{0}^{t}\mathrm{e}^{\alpha \tau}\int\phi_{xx}^4&\leq	\int_{0}^{t}\mathrm{e}^{\alpha \tau}\cdot\|\phi_{xx}\|_{L^{\infty}}^2\left(\|\phi_{xx}\|_{L^{2}}^2+\|\psi_{x}\|_{L^{2}}^2\right)\\&	\leq C \delta_1\int_{0}^{t}\mathrm{e}^{\alpha \tau} \left(\|\phi_{xx}\|_{L^{2}}\|\phi_{xxx}\|_{L^{2}} +\|\phi_{xx}\|_{L^{2}}^2\right)\\&\leq C\delta_1 \int_{0}^{t}\mathrm{e}^{\alpha \tau}\int \phi_{xxx}^2+C\int_{0}^{t}\mathrm{e}^{\alpha \tau}\int \phi_{xx}^2.
		\end{aligned}
	\end{equation}
	We thus update \eqref{eq21} as
	\begin{equation}\label{int_0^t phixxx}
		\begin{aligned}
			\int_{0}^{t}\mathrm{e}^{\alpha \tau}\int \phi_{xxx}^2\leq C\int_{0}^{t}\mathrm{e}^{\alpha \tau}\int\left(\phi_{x}^2+\psi_{x}^2 +\phi_{xx}^2+\psi_{xx}^2+\phi_{xt}^2\right)\leq C\left(\left\|\phi_0\right\|_{H^2}^2+\left\|\psi_0\right\|_{H^1}^2\right),
		\end{aligned}
	\end{equation}
	where we have used \eqref{together}, \eqref{int_0^t phixx}-\eqref{int_0^t psixx} in the last inequality. Combining \eqref{together}, \eqref{int_0^t phixx}, \eqref{int_0^t psixx} and \eqref{int_0^t phixxx}, we
	then obtain \eqref{H2 eistimate} and finish the proof of Lemma \ref{lemma H2}.
\end{proof}

\begin{proof}[Proof of Proposition \ref{proposition a priori estimates}] The desired estimates \eqref{a priori estimate1}-\eqref{a priori estimate3} follow from Lemmas \ref{lemma L2}-\ref{lemma H2}.
\end{proof}

\subsection{Proof of Theorem \ref{stability theorem}}
Recalling from the proof of Lemma \ref{lemma L2}, we have selected $\delta_{1}$ such that $\delta_{1}\leq\delta_0$, where $\delta_0$ is taken arbitrarily and mentioned in Proposition \ref{pro local existence}. Choosing $\left\|\phi_0\right\|_{H^2}^2+\left\|\psi_0\right\|_{H^1}^2$ small enough such that
	\begin{equation}
		\left\|\phi_0\right\|_{H^2}^2+\left\|\psi_0\right\|_{H^1}^2\leq \frac{\delta_1}{2}\text{ and }C\left( \left\|\phi_0\right\|_{H^2}^2+\left\|\psi_0\right\|_{H^1}^2\right)\leq	 \frac{\delta_1}{2},
	\end{equation}
	which implies $\left\|\phi_0\right\|_{H^2}^2+\left\|\psi_0\right\|_{H^1}^2\leq \epsilon:=\min\left\{\frac{\delta_1}{2},\frac{\delta_1}{2C} \right\}$. Since $\left\|\phi_0\right\|_{H^2}^2+\left\|\psi_0\right\|_{H^1}^2\leq \epsilon\leq \delta_1/2<\delta_0$, then Proposition \ref{pro local existence} guarantees the existence of a unique solution $(\phi,\psi)\in X(0,T_0)$ for the system \eqref{phipsi equ} satisfying
	\begin{equation}\nonumber
		\|\phi(\cdot, t)\|_{H^2}^2+\|\psi(\cdot, t)\|_{H^1}^2 \leq 2\left(\|\phi_0\|_{H^2}^2+\|\psi_0\|_{H^1}^2 \right)\leq2\epsilon\leq2\delta_1/2=\delta_1,
	\end{equation}
	for any $0\leq t\leq T_0$. Subsequently, by employing Proposition \ref{proposition a priori estimates}, we establish that
	\begin{equation}\nonumber
		\|\phi(\cdot, T_0)\|_{H^2}^2+\|\psi(\cdot, T_0)\|_{H^1}^2 \leq C\left(\|\phi_0\|_{H^2}^2+\|\psi_0\|_{H^1}^2 \right)\leq C \epsilon\leq C \delta_{1}/2C=\delta_1/2<\delta_0.
	\end{equation}
	Now considering the system \eqref{phipsi equ} with the \lq\lq initial data\rq\rq at $T_0$, and utilizing Proposition \ref{pro local existence} once again, we obtain a unique solution $\phi\in X(T_0,2T_0)$, eventually on the interval $[0,2T_0]$, satisfying
	\begin{equation}\nonumber
		\|\phi(\cdot, t)\|_{H^2}^2+\|\psi(\cdot, t)\|_{H^1}^2 \leq 2\left( \|\phi(\cdot, T_0)\|_{H^2}^2+\|\psi(\cdot, T_0)\|_{H^1}^2\right)\leq2\delta_1/2=\delta_1,
	\end{equation}
	for any $0\leq t\leq 2T_0$. Then, applying Proposition \ref{proposition a priori estimates} with $T=2T_0$ again, we deduce that
	\begin{equation}\nonumber
		\|\phi(\cdot, 2T_0)\|_{H^2}^2+\|\psi(\cdot, 2T_0)\|_{H^1}^2 \leq C\left(\|\phi_0\|_{H^2}^2+\|\psi_0\|_{H^1}^2 \right)\leq C \epsilon\leq C \delta_{1}/2C=\delta_1/2<\delta_0.
	\end{equation}
	Hence, repeated using the result of local existence and Proposition \ref{proposition a priori estimates} and the standard extension procedure, we obtain the global well-posedness of \eqref{phipsi equ}-\eqref{phipsi boundary condition} in $X(0,\infty)$. Thanks to the decomposition \eqref{decompose},  we conclude that the problem \eqref{original problem}-\eqref{original boundary condition} admits a unique global solution $(u,v)$ satisfying
\begin{equation}\nonumber
	u-U \in C\left([0, \infty) ; H^1\right) \cap L^2\left(0, \infty ; H^2\right), \quad v-V \in C\left([0, \infty) ; H^1\right) \cap L^2\left(0, \infty ; H^2\right).
\end{equation}
In view of \eqref{decompose}, by H\"{o}lder's inequality and Sobolev inequality, we have
\[\|\phi_0\|_{L^2}\leq C\|\phi_0\|_{L^{\infty}}\leq C\|\phi_{0x}\|_{L^2}\leq C\|u_0-U\|_{H^1}.\] It then follows from Proposition \ref{global existence} that
\begin{equation*}
	\|(u-U, v-V)(\cdot, t)\|_{H^1} \leq C\left\|(u_0-U,v_0-V)\right\|_{H^1}\mathrm{e}^{-\frac{\alpha}{2} t}.
\end{equation*}
We thus finish the proof of Theorem \ref{stability theorem}.

\section{Numerical simulations}\label{numerical}
In this section, we carry out some numerical simulations for the system \eqref{original problem}-\eqref{original boundary condition} with the initial condition $0\leq u_0<1$.

We set $q(u)=(1-u)^+$ (where the crowing capacity $K=1$), $\chi=20$, $L=1$ and $b=1$ of \eqref{original problem}-\eqref{original boundary condition}, resulting in the following model:
\begin{equation}\label{transform problem}
	\begin{cases}u_t=\left[\left(q(u)-q^{\prime}(u) u\right) u_x-20 q(u) u v_x\right]_x, & x \in(0,1), \\ v_t=v_{x x}-u v, & x \in(0,1),\\ \left( u,v\right)\left(x_0,0\right)=\left(u_0,v_0\right)(x),& x \in(0,1)\end{cases}
\end{equation}
with boundary conditions
\begin{equation}\label{transform boundary condition}
	\begin{cases}
		u_x(0,t)=0,~\left((q(u)-q^{\prime}(u) u) u_x-20 q(u) u v_x\right)(1,t)=0, \\
		v_x(0,t)=0,~v(1,t)=1.
\end{cases}\end{equation}
Let $u_0(x)=\frac{1}{2}x^2(3 -2x)$ and $v_0(x)=x^2$. Then in view of \eqref{m}, we get $m=\frac{1}{4}$. To numerically solve the problem, we set the computational domain $(x,t)\in[0,1]\times [0,100]$ with mesh $\Delta x=0.005$ and $\Delta t=0.0005$.
\begin{figure}
	\begin{center}
        \includegraphics[width=7cm]{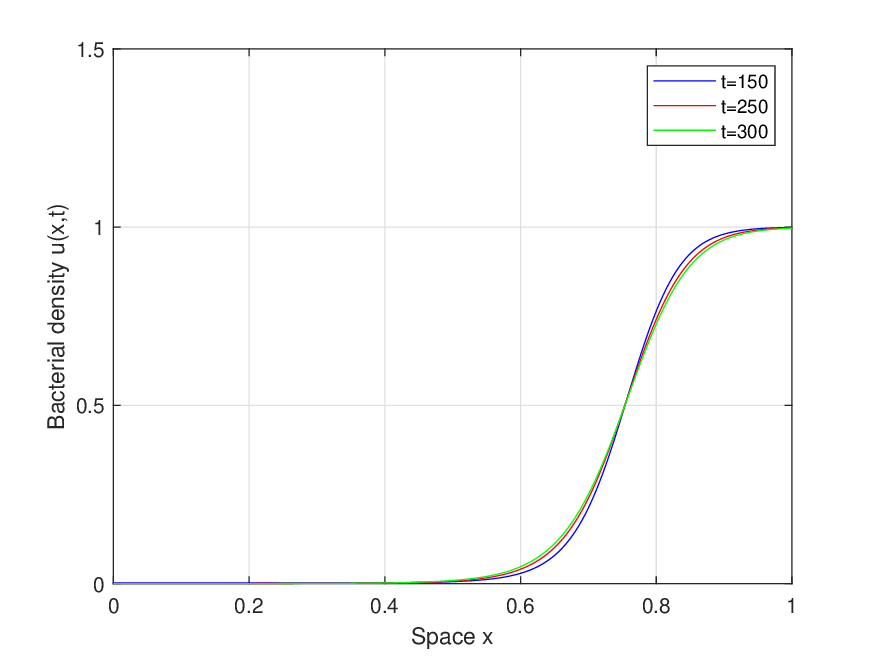}
		\includegraphics[width=7cm]{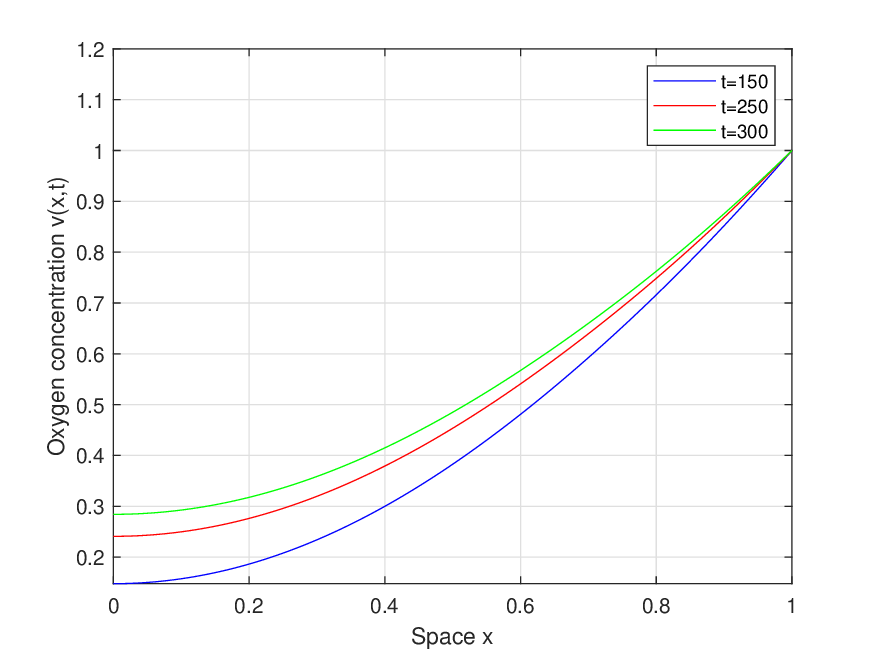}
		\caption{The numerical solution $(u,v)(x,t)$ with $0\leq u_0<1$ at $t=150, 250, 300$, where $q(u)=(1-u)^+$, $\chi=20$, $L=1$ and $b=1$. }
		\label{fig1}
	\end{center}
\end{figure}
Figure \ref{fig1} depicts the solution of system \eqref{transform problem}-\eqref{transform boundary condition} at various time points, revealing a sharp rise of $u(x,t)$ around $x=\frac{3}{4}$, which indicates the convergence to the steady state solution $(U,V)$ as $t\rightarrow \infty$. These numerical simulations illustrate and confirm our theoretical results.

\section*{Acknowledgments}
The authors are grateful to the referee for the insightful comments and suggestions, which
lead to great improvements of our original manuscript. This work is supported by the National Natural Science Foundation of China (No. 12371216) and the Natural Science Foundation of Jilin Province (No. 20210101144JC).



\begin{thebibliography}{99}

\bibitem{Adler} Adler J. Chemotaxis in bacteria. Science, 1966, 153: 708--716.


\bibitem{JinHY} Fan L, Jin H Y. Global existence and asymptotic behavior to a chemotaxis system with
consumption of chemoattractant in higher dimensions. J Math Phys, 2017, 58(1): 011503.

\bibitem{Lankeit-Long-time} Fuest M, Lankeit J, Mizukami M. Long-term behaviour in a parabolic-elliptic chemotaxis-consumption
model. J Differ Equ, 2021, 271: 254--279.

\bibitem{Hillen-Painter01} Hillen T, Painter K. Global existence for a parabolic chemotaxis model with prevention of overcrowding. Adv Appl Math, 2001, 26(4): 280--301.

\bibitem{Hillen-Painter09} Hillen T, Painter K. A user's guide to PDE models for chemotaxis. J Math Biol, 2009, 58(1): 183--217.

\bibitem{one-dimension-S} Hong G, Wang Z. Asymptotic stability of exogenous chemotaxis systems with physical boundary conditions. Quart Appl Math, 2021, 79(4): 717--743.


\bibitem{Jiang-Zhang} Jiang J, Zhang Y. On convergence to equilibria for a chemotaxis model with volume-filling effect. Asymptot Anal, 2009, 65(1-2): 79--102.


\bibitem{KS} Keller E, Segel L. Models for chemtoaxis. J Theor Biol, 1971, 30(2): 225--234.


\bibitem{Lai16} Lai X, Chen X, Wang M, Qin C, Zhang Y. Existence, uniqueness, and stability of bubble solutions of a chemotaxis model. Discrete Contin Dyn Syst, 2016, 36(2): 805--832.

\bibitem{Lankeit-2017} Lankeit J, Wang Y. Global existence, boundedness and stabilization in a high-dimensional chemotaxis system with consumption. Discrete Contin Dyn Syst, 2017, 37(12): 6099--6121.

\bibitem{LAK} Lauffenburger D, Aris R, Keller K. Effects of cell motility and chemotaxis on microbial population growth. Biophys J, 1982, 40(3): 209--219.

\bibitem{Lee} Lee C, Wang Z, Yang W. Boundary-layer profile of a singularly perturbed nonlocal semi-linear problem arising in chemotaxis. Nonlinearity, 2020, 33(10): 5111--5141.

\bibitem{Li-Li} Li X, Li J. Stability of stationary solutions to a multidimensional parabolic-parabolic chemotaxis-consumption model. Math Models Methods Appl Sci, 2023, 33(14): 2879--2904.


\bibitem{Ma-Gao-Tong-Han} Ma M, Gao M, Tong C, Han Y. Chemotaxis-driven pattern formation for a reaction-diffusion-chemotaxis model with volume-filling effect. Comput Math Appl, 2016, 72(3): 1320--1340.


\bibitem{Ma-Ou-Wang12} Ma M, Ou C, Wang Z. Stationary solutions of a volume-filling chemotaxis model with logistic growth and their stability. SIAM J Appl Math, 2012, 72(3): 740--766.

\bibitem{Ma-Wang15} Ma M, Wang Z. Global bifurcation and stability of steady states for a reaction-diffusion-chemotaxis model with volume-filling effect. Nonlinearity, 2015, 28(8): 2639--2660.


\bibitem{Ma-Wang17} Ma M, Wang Z. Patterns in a generalized volume-filling chemotaxis model with cell proliferation. Anal Appl, 2017, 15(01): 83--106.


\bibitem{Ou-Yuan09} Ou C, Yuan W. Traveling wavefronts in a volume-filling chemotaxis model. SIAM J Appl Dyn Syst, 2009, 8(1): 390--416.


\bibitem{Hillen-Painter02} Painter K, Hillen T. Volume-filling and quorum-sensing in models for chemosensitive movement. Canadian Appl Math Quart, 2002, 10(4): 501--543.

\bibitem{Potapov-Hillen05} Potapov A, Hillen T. Metastability in chemotaxis models. J Dynam Differential Equations, 2005, 17: 293--330.


\bibitem{TaoY2011} Tao Y. Boundness in a chemotaxis model with oxygen consumption by bacteria. J Math Anal Appl, 2011, 381(2): 521--529.


\bibitem{TaoY2012} Tao Y, Winkler M. Eventual smoothness and stabilization of large-data solutions in a three dimensional
chemotaxis system with consumption of chemoattractant. J Differ Equ, 2012, 252(3): 2520--2543.


\bibitem{Tuval} Tuval I, Cisneros L, Dombrowski C, Wolgemuth C, Kessler J, Goldstein R. Bacterial swimming and oxygen transport near contact lines. Proc Natl Acad Sci, 2005, 102(7): 2277--2282.

\bibitem{Wang2000} Wang X. Qualitative behavior of solutions of chemotactic diffusion systems: effects of motility and chemotaxis and dynamics. SIAM J Math Anal, 2000, 31(3): 535--560.



\bibitem{Wang-Xu13} Wang X, Xu Q. Spiky and transition layer steady states of chemotaxis systems via global bifurcation and Helly's compactness theorem. J Math Biol, 2013, 66(6): 1241--1266.

\bibitem{Wang-Hillen07} Wang Z, Hillen T. Classical solutions and pattern formation for a volume filling chemotaxis model. Chaos, 2007, 17(3): 037108.


\bibitem{Wang-Winkler-Wrzosek11} Wang Z, Winkler M, Wrzosek D. Singularity formation in chemotaxis systems with volume-filling effect. Nonlinearity, 2011, 24(12): 3279--3297.

\bibitem{Wang-Winkler-Wrzosek12} Wang Z, Winkler M, Wrzosek D. Global regularity versus infinite-time singularity formation in a chemotaxis model with volume-filling effect and degenerate diffusion. SIAM J Math Anal, 2012, 44(5): 3502--3525.

\bibitem{Wrzosek04} Wrzosek D. Global attractor for a chemotaxis model with prevention of overcrowding. Nonlinear Anal, 2004, 59(8): 1293--1310.

\bibitem{Wrzosek06} Wrzosek D. Long-time behaviour of solutions to a chemotaxis model with volume-filling effect. Proc Roy Soc Edinburgh Sect A, 2006, 136(2): 431--444.


\bibitem{Wrzosek10} Wrzosek D. Volume filling effect in modelling chemotaxis. Math Model Nat Phenom, 2010, 5(1): 123--147.


\end{thebibliography}
\end{document}